\documentclass{article}

\usepackage{lipsum}
\usepackage{amsfonts}
\usepackage{graphicx}
\usepackage{epstopdf}
\usepackage{algorithmic}
\ifpdf
  \DeclareGraphicsExtensions{.eps,.pdf,.png,.jpg}
\else
  \DeclareGraphicsExtensions{.eps}
\fi

\usepackage{amsopn}
\usepackage{amsmath,amsfonts,amssymb,enumitem}
\usepackage{xcolor}
\usepackage[algo2e,linesnumbered,vlined,algoruled]{algorithm2e}

\usepackage{placeins}
\SetCommentSty{mycommfont}

\newcommand{\Hdef}{\mathfrak{H}} 
\newcommand{\Act}[1]{\mathbb{A} \mathopen{} \left( #1 \right)\mathclose{}}    
\newcommand{\Gen}{\mathbb{G}}    
\newcommand{\bigC}{C}    

\newcommand{\cS}{\mathcal{S}}    
\newcommand{\cL}{\mathcal{L}}    
\newcommand{\cB}{\mathcal{B}}      
\newcommand{\cLmax}{\mathcal{L}_{\rm max}}      
\newcommand{\Deltamax}{\Delta_{\rm max}}
\newcommand{\etamax}{\eta_{\rm max}}
\newcommand{\R}{\mathbb R}       
\newcommand{\defined}{\triangleq}
\newcommand{\minimize}{\operatornamewithlimits{minimize}}

\newcommand{\trans}{{\mkern-1.5mu\mathsf{T}}}

\newcommand{\kappaief}{\kappa_{i, {\rm ef}}}
\newcommand{\kappaieg}{\kappa_{i, {\rm eg}}}
\newcommand{\kappaibmh}{\kappa_{i, {\rm mH}}}
\newcommand{\kappabmh}{\kappa_{{\rm H}}}
\newcommand{\kappaf}{\kappa_{\rm f}}
\newcommand{\kappag}{\kappa_{\rm g}}

\newcommand{\partialC}{\partial_{\rm C}}
\newcommand{\Lh}{L_h}
\newcommand{\kappad}{\kappa_{\rm d}}
\newcommand{\gammad}{\gamma_{\rm d}}
\newcommand{\gammai}{\gamma_{\rm i}}
\newcommand{\linerefa}[1]{Line~\ref{line:#1}}
\newcommand{\linerefb}[1]{Line~\hyperref[line:#1]{P.\ref{line:#1}}}
\newcommand{\cH}{\mathcal{H}} 
\newcommand{\LnF}{L_{\nabla \! F}}
\newcommand{\LnFi}{L_{\nabla \! F_i}}

\newcommand{\ds}{\displaystyle}
\newcommand{\jtext}[1]{[ \text{\footnotesize #1}]}
\newcommand{\tn}{\textnormal}
\newcommand{\Lnorm}[1]{\ell_{#1}} 
\newcommand{\Proj}{\mathbf{proj}}
\newcommand{\cop}[1]{\mathbf{co}\left(#1\right)} 
\newcommand{\clp}[1]{\mathbf{cl}\left( #1 \right)} 
\newcommand{\intp}[1]{\mathbf{int}\left( #1 \right)} 
\renewcommand{\Im}[1]{\mathop{\mathbf{Im}}\left( #1 \right)} 
\newcommand{\hdom}{\cop{\Im{\cLmax}}} 

\newcommand{\thmref}[1]{{\rm Theorem~\ref{thm:#1}}}
\newcommand{\figref}[1]{{\rm Figure~\ref{fig:#1}}}

\newcommand{\lemref}[1]{{\rm Lemma~\ref{lem:#1}}}
\newcommand{\defref}[1]{{\rm Definition~\ref{def:#1}}}
\newcommand{\assref}[1]{{\rm Assumption~\ref{ass:#1}}}
\newcommand{\assreftwo}[2]{{\rm Assumptions~\ref{ass:#1} and \ref{ass:#2}}}
\newcommand{\assrefrange}[2]{{\rm Assumptions~\ref{ass:#1}--\ref{ass:#2}}}
\newcommand{\subassref}[2]{{\rm Assumption~\ref{ass:#1}.\ref{subass:#2}}}
\newcommand{\algref}[1]{{\rm Algorithm~\ref{alg:#1}}}
\newcommand{\secref}[1]{Section~\ref{sec:#1}}

\usepackage{fullpage,amsthm}
\usepackage[breaklinks=true, colorlinks=true, linkcolor={blue!90!black}, citecolor={blue!90!black}, urlcolor={blue!90!black}]{hyperref}
\newtheorem{definition}{Definition}[section]
\newtheorem{theorem}{Theorem}[section]
\newtheorem{lemma}{Lemma}[section]
\newtheorem{remark}{Remark}[section]

\newcommand{\email}{\url}
\usepackage{algorithm}

\newtheorem{assump}{Assumption}

\renewcommand{\frac}{\tfrac}
\graphicspath{{../../images/}}

\usepackage[normalem]{ulem} 
 
\DeclareMathOperator*{\argmax}{arg\,max}

\usepackage{mathtools}

\title{Manifold Sampling for Optimizing Nonsmooth Nonconvex Compositions}

\author{Jeffrey Larson\thanks{Mathematics and Computer Science Division, Argonne National Laboratory, Lemont, IL 60439 (\email{jmlarson@anl.gov}, \email{mmenickelly@anl.gov}).}
\and Matt Menickelly\footnotemark[1]
\and Baoyu Zhou\thanks{Department of Industrial and Systems Engineering, Lehigh University, Bethlehem, PA 18015 (\email{baoyu.zhou@lehigh.edu}).}
}

\begin{document}

\maketitle

\begin{abstract}
  We propose a manifold sampling algorithm for minimizing a nonsmooth
  composition $f= h\circ F$, where we assume $h$ is nonsmooth and may
  be inexpensively computed in closed form and $F$ is smooth but its
  Jacobian may not be available. We additionally assume that the
  composition $h\circ F$ defines a continuous selection. Manifold
  sampling algorithms can be classified as model-based derivative-free
  methods, in that models of $F$ are combined with particularly
  sampled information about $h$ to yield local models for use within a
  trust-region framework. We demonstrate that cluster points of the
  sequence of iterates generated by the manifold sampling algorithm
  are Clarke stationary.  We consider the tractability of three
  particular subproblems generated by the manifold sampling algorithm
  and the extent to which inexact solutions to these subproblems may
  be tolerated. Numerical results demonstrate that manifold sampling
  as a derivative-free algorithm is competitive with state-of-the-art
  algorithms for nonsmooth optimization that utilize first-order
  information about $f$.
\end{abstract}



\section{Introduction}
We consider
unconstrained composite optimization problems of the form
\begin{equation} \label{eq:func_def}
  \minimize\left\{ f(x)\colon x\in\R^n \right\} \mbox{  when  }  f(x) \defined 
  h(F(x)),
\end{equation}
where 
$h \colon \R^p \to \R$ is a structured, possibly nonsmooth and nonconvex function
but its 
Clarke subdifferential $\partialC h(z)$ is known at any $z \in
\mathbb{R}^p$ in the domain of $h$
and the function $F \colon \R^n\to \R^p$, where $F =[F_1~F_2~\cdots~F_p]^\top$, is assumed
continuously differentiable. We are especially motivated by problems where $F$ is
expensive to evaluate and the Jacobian $\nabla F(x)$ is assumed unavailable.
Such problems arise, for example, when using a nonsmooth loss function
$h$ to measure the quality of output from a function $F$ that depends on some
expensive simulation. 

\emph{Manifold sampling} is an approach for solving~\eqref{eq:func_def} that
constructs
models of $F$ and
combines these models using particular sampled information about $h$. 
Manifold sampling was originally developed for the specific case where $h$
is the $\ell_1$-norm~\cite{LMW16} and was later generalized to the case where $h$ is
a potentially nonconvex \emph{continuous selection} of affine functions~\cite{KLW18}. 
\begin{definition}\label{def:continuous_selection}

  A function $h$ is a \emph{continuous selection} on $U\subseteq\R^p$ if it is
  continuous on $U$ and
  $
  h(z) \in \{h_j(z)\colon h_j \in\Hdef\}$ for all $z \in U$,
  where $h_j:\R^p \to \R$ and $\Hdef$ is a finite set of \emph{selection functions}.
\end{definition}

While the manifold sampling algorithm presented in~\cite{KLW18} required
 every $h_j\in\Hdef$ to be an affine function, this work considers the 
case where each $h_j\in \Hdef$ is assumed only to be Lipschitz continuous and Lipschitz
continuously differentiable. Past manifold sampling algorithms made the
implicit assumption that each selection function was uniquely represented by
its gradient; such an assumption does not hold in this more general case.

As in previous manifold sampling algorithms, 
the unavailability of $\nabla F(x)$ at a given $x \in \mathbb{R}^n$ is handled through standard techniques of 
model-based derivative-free
optimization, that is, via the construction of models $m^{F_i}(x)$ of $F_i(x)$.  
We then employ
\begin{equation}\label{eq:Mdef}
M(x) \defined \begin{bmatrix} m^{F_1}(x)  \cdots m^{F_p}(x) \end{bmatrix}^\trans  \text{ and } \nabla M(x) \defined \begin{bmatrix} \nabla m^{F_1}(x)  \cdots  \nabla m^{F_p}(x) \end{bmatrix}
\end{equation}
as approximations of $F(x)$ and $\nabla F(x)$, respectively.  
In the present work, we will use not only
knowledge of the functions $h_j \in \Hdef$ that define $h$ but also
$\nabla h_j(z)$, the gradients of linearizations of $h_j$ at points $z$, in order to
update iterates.  For each continuously differentiable $h_j\in\Hdef$, 
the composition $h_j(F(x))$ is also continuously differentiable, and from a simple
application of the chain rule 
$$\nabla h_j(F(x)) = \nabla F(x) \nabla h_j(F(x))
\approx \nabla M(x) \nabla h_j(F(x)).$$
Thus, when considering topologically connected point sets satisfying
\begin{equation}
\label{eq:manifold} 
\{x\in\R^n\colon h(F(x))=h_j(F(x))\},
\end{equation} 
we can use standard derivative-free techniques to produce sufficiently accurate models
of $h_j\circ F$ and hence sufficiently accurate models of $h\circ F$ when
restricted to~\eqref{eq:manifold}.
\emph{Manifolds}\footnote{
Using a result such as Theorem 2 of~\cite{Rockafellar2003}, one can show that
functions $h$ satisfying \defref{continuous_selection} admit a Whitney
stratification, and the resulting strata could be interpreted as manifolds in \eqref{eq:manifold}.
See also~\cite{Bolte2007,Davis2020,Ioffe2009}. Accounting for such a
manifold representation would be cumbersome, however, and hence we use the notation of 
continuous selections.}
in this manuscript refer to any topologically connected set of the
form~\eqref{eq:manifold} and hence refer to regions where $h_j\circ M$ might
be used as models of $f = h\circ F$. 
In each manifold sampling iteration, we construct models of $F$ and
make use of $h_j\circ F$ for some $h_j\in \Hdef$.  These models are then
assembled in a particular manner to suggest search directions within an
iterative (trust-region) method.

Our work is not in the scope of optimization of functions defined on
 general Riemannian manifolds, which is an active, but distinct, area of research.

\subsection{Literature review}

Unconstrained nonsmooth \emph{noncomposite convex} optimization given a subgradient oracle has become a
classical method;
 see, for instance,~\cite{rockafellar2009variational}
for a textbook treatment of first-order subgradient methods. 
A historically popular class of methods for the solution for convex nonsmooth optimization methods has been
bundle methods; see~\cite{makela2001} for a survey from 2001. 
Given a \emph{noncomposite nonconvex} objective function but still assuming access to a subgradient oracle, 
various solution methodologies have been proposed.  
Subgradient methods tailored to the nonconvex setting are investigated in, for instance,~\cite{Bagirov2008a, Bagirov2013},
while bundle methods suitable for the nonconvex setting have also been developed (see, e.g.,~\cite{hare2010redistributed, hare2016proximal, Kiwiel96}). 
The approach of~\cite{pang1991minimization} iteratively constructs convex
second-order models of the objective and employs a line search for
globalization; in~\cite{qi1994trust} similar models of the objective are constructed,
but a trust region is employed for globalization, yielding convergence of a
subsequence to a Dini stationary point. Difference-of-convex approaches for
nonsmooth composite objectives 
have also been studied in recent works~\cite{Cui2018,Liu2020}.

For unconstrained nonsmooth \emph{noncomposite nonconvex} optimization when in the \emph{derivative-free}
setting, Bagirov, Karas\"{o}zen, and Sezer~\cite{Bagirov2007} proposed the so-called discrete gradient method,
which computes
 approximate subgradients for use in a subgradient descent framework; see also~\cite{riis2018geometric}.
Direct-search methods in derivative-free optimization have been historically
concerned with convergence to Clarke stationary points and are hence suitable for nonsmooth optimization.
See the book~\cite{AudetHare2017}
for an excellent treatment of this subject. 
Bundle methods, as well as trust-region bundle methods, have also been considered in the 
derivative-free setting; see, for instance,~\cite{karmitsa2012limited,liuzzi2019trust}.  

Unconstrained nonsmooth \emph{composite} optimization
problems of the form \eqref{eq:func_def} have been given special
attention in the literature. Works from the 1980s 
provide fundamental analyses for the
case where $h$ is convex and $\nabla F$ is available~\cite{fletcher1982model, fletcher1982second,
womersley1986algorithm,
yuan1985conditions, yuan1985superlinear}.
In a \emph{derivative-free} setting, the authors in~\cite{garmanjani2016trust,
grapiglia2016derivative} analyze algorithms for composite optimization where $h$ is a general convex function but
$\nabla F$ is not available.  To the best of our knowledge,~\cite{KLW18} is the
only work that allows for $h$ to be nonconvex and does not require access to $\nabla F$.  

The case of $h$ in \eqref{eq:func_def}
 being a general convex function has enjoyed special attention~\cite{Cartis2011, fletcher1982model, fletcher1982second,
garmanjani2016trust, grapiglia2016derivative, yuan1985conditions}.
Convex functions $h$ are natural in many applications.
For instance, they are frequently used
as penalty functions, such as $\|\cdot\|_p$, 
and $\|\max(\cdot,0)\|_p$ (an exact penalty for inequality
constraints).
Thus, in the literature one sometimes sees treatments of particular convex functions $h$.
For example, in~\cite{CFMR2018, Wild14}, $h$ is fixed as $\|\cdot\|_2^2$,
while in ~\cite{Hare2017b, menickelly2020, womersley1986algorithm}, $h$ is fixed
as $\max(\cdot)$.
In~\cite{fletcher1982second, yuan1985superlinear}, $h$ is fixed as a polyhedral convex function,
which is a convex piecewise-affine function. The work in~\cite{KLW18} removes the convexity requirement to 
address functions $h$ that are assumed only to be continuous and piecewise-affine. 

We note that the 
algorithmic differentiation (AD) community has analyzed methods for computing
generalized derivatives of piecewise-smooth functions for use in gradient-based
nonsmooth optimization methods.
The authors of~\cite{Griewank2016a} consider the ``abs-normal form" (see, e.g.,~\cite{Griewank2016b, Walther2019})
of local nonsmooth models of a class of piecewise-smooth functions 
that commonly appear in computer codes, and they 
demonstrate how standard AD tools can be extended to
derive these models, which are then used to design bundle-type methods. The
approach in~\cite{Mitsos2009} considers a limited class of composite functions
and uses a
forward-AD method to compute a subgradient of such functions for use in
McCormick relaxations, which can be solved by global optimization methods;
the resulting relaxation is more amenable to global optimization methods. 
A state-of-the-art forward-AD mode for computing generalized
derivatives of composite piecewise-differentiable functions is presented
in~\cite{Khan2015}. 
A reverse-AD method for computing the subgradients of the same class of
functions as considered in~\cite{Mitsos2009} is derived in~\cite{beckers2012}.
In~\cite{Khan2013}, the authors demonstrate a practical reverse-AD mode for
computing generalized derivatives of composite piecewise-differentiable
functions. 

Gradient sampling methods~\cite{burke2020gradient,burke2002approximating,burke2005robust,
curtis2013adaptive, Kiwiel2007} are designed for noncomposite nonsmooth nonconvex optimization.
We pay special attention to them because they are closest in
spirit to manifold sampling algorithms. 
Gradient sampling methods were originally designed~\cite{burke2002approximating,burke2005robust} for
the minimization of locally Lipschitz functions, a broader 
class of nonsmooth functions than those analyzed in the present paper.
Gradient sampling methods compute random samples of gradients at multiple points in an
$\epsilon$-neighborhood of a current point $x$ in order to approximate the
Clarke $\epsilon$-subdifferential of $f(x)$, $\partial_{\epsilon} f(x)$.
In contrast to gradient sampling, we note that 
manifold sampling does not require access to gradient information and does not rely on
randomization.
By exploiting knowledge of the structure of the objective in~\eqref{eq:func_def}
(in particular, the finiteness of the set $\Hdef$ defining $h$), 
manifold sampling does not depend on random sampling to identify the presence of 
distinct manifolds in a neighborhood of $x$.  
We note that Kiwiel~\cite{KIWIEL2010} proposed a gradient sampling method for the
derivative-free setting by computing approximate (finite-difference) gradients,
but the method still depends on randomization.

\subsection{Real-world example}
While convex or piecewise-linear forms of $h$ may be common, they are
far from exclusive. As one example, particle beamline scientists often seek
operational parameters that produce a tightly bunched beam at some point in
space. This allows a sample to be placed at this point in space in order to be
hit by the tightest-possible beam. The spread of the beam is measured by the
\emph{normalized emittance}. A beamline simulation is run for a given set of
operational parameters $x$, producing three vectors of simulation output $F_{1,j}(x)$, $F_{2,j}(x)$,
and $F_{3,j}(x)$ for each position $j$ in a finite set $J$. The objective is
to minimize $\underset{j \in J}{\min} \sqrt{F_{1,j}(x) F_{2,j}(x) - F_{3,j}(x)^2}$.
See~\cite{Wiedemann2015} for greater detail.
In this case, we have a nonconvex, not piecewise-linear, nonsmooth function
$
h(z)
\defined \underset{j \in J}{\min} \left\{ h_j(z) \defined \sqrt{z_{1,j} z_{2,j}
- z_{3,j}^2} \right\}.
$
\subsection{Notation and definitions}
All norms are assumed to be $\Lnorm{2}$ norms.
The \emph{closure}, \emph{interior}, and \emph{convex hull} of a set
$\cS$ are denoted $\clp{\cS}$, $\intp{\cS}$, and
$\cop{\cS}$, respectively. 
The \emph{image} of a set $\cS$ through $F$ is 
$\Im{\cS} \defined \{ F(x)\colon x\in\cS \}$.
We define $\cB(x;\Delta) = \{ y\colon \left\| x - y \right\| \le \Delta  \}$. 
In the rest of this paper, we use $0$ to denote both a scalar and a
vector with a zero in each entry, depending on context.
We define $e$ to be the vector with a one in each entry.

We say a function $f$ is \emph{Lipschitz continuous with constant $L_f$ on $\Omega \subset
\mathbb{R}^n$} if $\left| f(x) - f(y) \right| \le L_f \left\| x-y \right\|$ for
all $x,y \in \Omega$. 
Recall that if $f$ is also continuously differentiable, then Lipschitz continuity immediately implies that
$\|\nabla f(x)\|\leq L_f$ for all $x\in \intp{\Omega}$.
Similarly, $f$ has a \emph{Lipschitz continuous gradient with
constant $L_{\nabla \! f}$ on $\Omega$} if 
$\left\| \nabla f(x) - \nabla f(y) \right\| \le L_{\nabla \! f} \left\| x-y \right\|$ for
all $x,y \in \Omega$.

The \emph{generalized Clarke subdifferential} of a locally Lipschitz
continuous function $f$ at a point $x$ is defined as 
$
  \partialC f(x) \defined \cop{\left\{ \lim_{y^j\to x} \nabla f(y^j)\colon  y^j \in \mathcal{D} \right\}},
  $
where $\mathcal{D}$ is the set of points where $f$ is differentiable.
That is, $\partialC f(x)$ is the convex hull of the set of all limiting
gradients from differentiable points that converge to $x$.
A point $x$ is called a \emph{Clarke stationary} point of $f$
if $0\in \partialC f(x)$.

\subsection{Organization}
\secref{assumptions} collects assumptions about $f$, $h$, $F$, and the models
approximating $F_i$. 
\secref{algorithm_overview} presents the essential components of 
the manifold sampling algorithm.
\secref{prelims} contains lemmas concerning the essential components of the 
manifold sampling algorithm that will be used in later convergence analysis. 
We highlight \lemref{candidate_search_lemma}
and \lemref{approx_solve}, which relate to the feasibility of two particular
subproblems encountered during each manifold sampling iteration.
In \secref{analysis2}, we demonstrate our main theoretical result, namely, that all cluster points of the sequence of iterates
produced by the manifold sampling algorithm are Clarke stationary.
\secref{tests} presents numerical experiments comparing an implementation of the proposed manifold
sampling algorithm with other methods for nonsmooth optimization, all of which are given
access to values of $\nabla f(x)$.
\secref{conclusion} contains some concluding remarks and discussion. 

\section{Problem Setting}\label{sec:assumptions}

We now present background material and assumptions.  
For an initial iterate $x^0 \in \mathbb{R}^n$, we first define $\cL(x^0) \defined \left\{ x\colon f(x) \le f(x^0) \right\}$.  Moreover,
for a constant $\Delta_{\max} > 0$ and point $x^0$, define 
\begin{equation}\label{eq:cLmaxdef}
\cLmax \defined \ds \bigcup_{x \in \cL(x^0)} \cB(x;\Delta_{\max}).
\end{equation}
We may now introduce some assumptions about the
objective function $f\colon\R^n\to\R$.

\begin{assump}\label{ass:f} We assume the following about $f$ and $F$.
  \begin{enumerate}[label=\tn{\textbf{\Alph*}.},ref=\tn{\Alph*},leftmargin=20pt]
    \item For a point $x^0 \in \R^n$, the set $\cL(x^0)$
     is bounded. \label{subass:level}
      
    \item Each $F_i$ 
     is Lipschitz continuous with 
      constant $L_{F_i}$ on $\cLmax$. \label{subass:Fval}

    \item Each $F_i$ 
    is Lipschitz gradient continuous with
      constant $\LnFi$ on $\cLmax$.\label{subass:Fgrad}
  \end{enumerate}
\end{assump}

From \subassref{f}{level}, we may conclude that $\cLmax$ is bounded.
In our analysis, we will demonstrate in \lemref{delta_to_0_linear} that all points evaluated by the manifold
sampling algorithm are contained in $\cLmax$.

We next define what it means for a given selection function $h_j$ to be
\emph{essentially active} in a continuous selection; see~\cite{Scholtes2012} for 
a deeper treatment of continuous selections. 

\begin{definition} \label{def:pc1manifold}
  Suppose $h$ is a continuous selection. We define
  \begin{align*}
    \cS_j \defined \left\{ z\colon h(z) = h_j(z) \right\}, \quad 
    \tilde{\cS}_j \defined \clp{\intp{\cS_{j}}}, \quad
    \Act{z} \defined \left\{ j\colon z \in \tilde{\cS}_j \right\}.
  \end{align*}
  We refer to elements of $\mathbb{A}(z)$ as \emph{essentially active indices}.
  We refer to any
  $h_j$ for which $j\in \mathbb{A}(z)$ as an \emph{essentially active selection
  function for $h$ at $z$}.  
  Moreover, the function $h_j$ is
  \emph{essentially active at $z\in\hdom$}, provided $h_j$ is an essentially active selection
  function for $h$ at $z$. 
 For a finite set $\mathbb{Z}$, let $\Act{\mathbb{Z}} = \ds\bigcup_{z \in \mathbb{Z}} \Act{z}$.
\end{definition}

Fundamentally, essentially active selection functions are those that describe
the behavior of $h$ near a point of interest. With this definition, we can make
the following assumptions on $h$ and the selection functions defining it.

\begin{assump}\label{ass:h} We assume the following about $h$.
  \begin{enumerate}[label=\tn{\textbf{\Alph*}.},ref=\tn{\Alph*},leftmargin=20pt]
    \item The function $h$ satisfies \defref{continuous_selection}. \label{subass:h}

    \item The set $\mathbb{A}(z)$ of essentially active indices for $h$ at $z$ is computable for
      any $z \in \hdom$. \label{subass:Ieh}
  \end{enumerate}
For all $z,z' \in \hdom$ and for each $h_j \in \Hdef$, we assume the following:
  \begin{enumerate}[label=\tn{\textbf{\Alph*}.},ref=\tn{\Alph*},leftmargin=20pt,resume]
    \item There exists $L_{h_j}$ such that $|h_j(z) - h_j(z')|\leq L_{h_j}\|z - z'\|$.  \label{subass:func}
    \item There exists $L_{\nabla h_j}$ such that $\|\nabla h_j(z) - \nabla h_j(z')\|\leq L_{\nabla h_j}\|z - z'\|$. \label{subass:grad}
  \end{enumerate}
\end{assump}

The models that we use to approximate the components $F_i$ of $F$ must be
sufficiently accurate. 
As is standard in model-based derivative-free optimization, we employ full
linearity as our standard of accuracy. 

\begin{definition} \label{def:flmodels}
  A function $m^{F_i} \colon \R^n \to \R$ is said to be a \emph{fully linear}
  model of $F_i$ on $\cB(x;\Delta)$ with constants $\kappaief$ and
  $\kappaieg$, provided 
  \begin{equation*}
    \begin{array}{rl}
      \left|F_i(x+s) - m^{F_i}(x+s)\right|\leq \kappaief \Delta^2 & \forall s\in \cB(0;\Delta),\\
      \left\|\nabla F_i(x+s) - \nabla m^{F_i}(x+s)\right\|\leq \kappaieg\Delta & \forall s\in \cB(0;\Delta).
    \end{array}
  \end{equation*}
\end{definition}

See, for instance,~\cite{LMW2019AN, Conn2009a} for a deeper treatment of
the construction of fully linear models. 
A fully linear model of a function can be constructed on $\cB(x;\Delta)$, for example, by
interpolating function values at $n+1$ well-poised points including $x$.
Note in this example, however, that we do not need $p \times (n+1)$
evaluations of $F$ in order to construct $p$ fully linear models $m^{F_i}$
since we assume that an evaluation of
$F(x)$ returns $F_1(x),\dots,F_p(x)$ simultaneously. 
Ultimately, the algorithm will use $\nabla m^{F_i}$ in
place of an unavailable $\nabla F_i$.  Each iteration of the manifold sampling algorithm
will ensure that the models $m^{F_i}$ are fully linear models of $F_i$ on
$\cB(x^k;\Delta_k)$, where $x^k$ and $\Delta_k$ correspond to the current point and
trust-region radius at iteration $k$.

\begin{assump}\label{ass:flmodels}
There exist constants $\{\kappa_{1,\mathrm{ef}},\dots,\kappa_{p,\mathrm{ef}}\}$
and $\{\kappa_{1,\mathrm{eg}},\dots,\kappa_{p,\mathrm{eg}}\}$, independent of $k$, such that
  for each component function $\{F_1, \ldots, F_p\}$ of $F$,
  each model $\{m^{F_1}, \ldots, m^{F_p}\}$ is
  fully linear on $\cB(x^k;\Delta_k)$ with the corresponding constants. 
  Moreover, each $m^{F_i}$ is twice continuously differentiable, and there exists $\kappaibmh$ so that $\|\nabla^2
  m^{F_i}(x)\|\leq \kappaibmh$ for all $x\in\cLmax$.  
\end{assump}

For ease of presentation and analysis, we define the following 
constants. 
\begin{definition}\label{def:constants}
  For the constants in \subassref{f}{Fgrad}, 
  \subassref{h}{func}, \defref{flmodels}, and \assref{flmodels}, define 
       $L_F \defined \sqrt{\sum_{i=1}^p L_{F_i}^2}$, 
       $\LnF \defined \sqrt{\sum_{i=1}^p \LnFi^2}$,
       $\Lh \defined \max_{j\in \{1,\ldots,\left| \Hdef \right|\}}\left\{ L_{h_j} \right\}$,
       $L_{\nabla h} \defined \max_{j\in \{1,\ldots,\left| \Hdef \right|\}}\left\{ L_{\nabla h_j} \right\}$,
       $ \kappaf \defined \sum_{i=1}^p \kappaief$, 
       $ \kappag \defined \sum_{i=1}^p \kappaieg$, 
       $ \kappabmh \defined \sum_{i=1}^p \kappaibmh$, and  $\bigC \defined (2L_h\kappag + 2L_{\nabla h}L_F^2 + L_h \LnF)$.
\end{definition}
Proposition~4.1.2 in~\cite{Scholtes2012} demonstrates that $L_h$ is in fact a
Lipschitz constant for $h$, and so our definition above is simply fixing a particular value
of the Lipschitz constant.

\section{Manifold Sampling for Piecewise-Smooth Compositions} \label{sec:algorithm_overview}
We now outline the essential components of the manifold sampling algorithm
(presented in \algref{manifoldsampling} in \secref{alg_state})
for solving problems of the form~\eqref{eq:func_def} satisfying \assref{f} and \assref{h}.
We draw special attention to \secref{master_model}, \secref{suff_dec}, and \secref{loop},
which respectively introduce three subproblems that must be solved in each iteration of the
manifold sampling algorithm.

\subsection{Sample set $\mathbb{Z}^k$, gradient set $\mathbb{D}^k$, and generator set $\Gen^k$}
 \label{sec:gensets}
Manifold sampling is an iterative method that
builds component models $m^{F_i}$ of each $F_i$ around the current point $x^k$.
We place the first-order terms of each model in
column $i$ of the matrix $\nabla M(x^k) \in
\R^{n \times p}$ as in~\eqref{eq:Mdef}.

As in past manifold sampling algorithms, $\nabla M(x^k)$
will be combined with gradients of selection functions to yield generalized
gradients for $f$ as in \eqref{eq:func_def}. 
But additional care must be taken when $h_j$ is assumed only to be smooth, and
not assumed to be piecewise-affine. 
Because the value of $\nabla h_j$ need not be unique on a given manifold, the
manifold sampling algorithm will maintain a
finite \emph{sample set} of points $\mathbb{Z}^k \subset \R^p$
representing points where different selection functions have been determined to be active.
In the algorithm, $\mathbb{Z}^k$ will include both the vector values of
$F(y)$ for previously evaluated points $y \in  \cB(x^k;\Delta_k)$ and points of the
form $\alpha F(x^k) + (1 - \alpha) F(y)$ for $\alpha \in [0,1]$. 
In other words, it may not always be the case that $\mathbb{Z}^k\subset\Im{\cLmax}$,
but it is always the case that $\mathbb{Z}^k\subset\hdom$. 

The sample set $\mathbb{Z}^k$ yields a \emph{set of gradients of linearizations
of active selection functions}, $\mathbb{D}^k$. Combining elements of
$\mathbb{D}^k$ with $\nabla M(x^k)$ produces a \emph{generator set}
$\Gen^k$ with elements of the form $\nabla M(x^k) \nabla h_j(z)$ for suitable
selection functions $h_j$ that are active within $\cB(F(x^k);L_F\Delta_k)$.
The set $\cop{\Gen^k}$ can then be treated  as a particular
approximation to $\partialC f(x^k)$. These sets are defined in the following.

\begin{definition}\label{def:useful_sets}
  Let $\mathbb{D}^k$ denote the set of gradients of linearizations of selection
  functions corresponding to the finite sample set $\mathbb{Z}^k \subset \mathbb{R}^p$. That is,\\
    $\mathbb{D}^k = \{ \nabla h_j(z)\colon z\in\mathbb{Z}^k,\; j\in\mathbb{A}(z) \}$.
  Let $\Gen^k$ denote the set of generators corresponding to the sample set $\mathbb{Z}^k$. That is,
	$\Gen^k = \{ \nabla M(x^k)\nabla h_j(z)\colon z\in\mathbb{Z}^k,\; j\in\mathbb{A}(z) \}$.
  Let $D^k$ be the matrix with columns that are the elements of $\mathbb{D}^k$. Let
  $G^k$ be the matrix with columns that are the elements of $\Gen^k$. That is, $G^k =
  \nabla M(x^k)D^k$.
\end{definition}

From \defref{useful_sets}, we see that a generator set $\Gen^k$ is a sample of approximate gradients
from various manifolds of the continuous selection
 that are potentially active at (or relatively near) $F(x^k)$. 
Ultimately, the minimum-norm
element of $\cop{\Gen^k}$, which is the projection of the origin to $\cop{\Gen^k}$, denoted by
\begin{equation} \label{eq:modelg_general}
  g^k\defined \Proj\left(0,
    \cop{\Gen^k}\right) \in \cop{\Gen^k},
\end{equation}
will be employed as the gradient of a smooth master model that will be minimized in a
trust-region $\cB(x^k;\Delta_k)$ to suggest trial points.
\lemref{weird_v_approx} will demonstrate that $\cop{\Gen^k}$ can approximate $\partialC
f(x^k)$ sufficiently well in order to guarantee that manifold
sampling converges to Clarke stationary points. 
When all the elements of $\mathbb{Z}^k$ are sufficiently close to $F(x^k)$
and $\|g^k\|\approx 0$, \lemref{weird_v_approx} suggests that $x^k$ is a Clarke
stationary point. This reasoning provides the rough roadmap for our analysis.

We see in \defref{useful_sets} that different choices of sample sets $\mathbb{Z}^k$ induce different
generator sets
$\Gen^k$. 
While the manifold sampling algorithm permits some flexibility in the selection of $\mathbb{Z}^k$,
our convergence results require some minimal assumptions on the construction of $\mathbb{Z}^k$.
\begin{assump} \label{ass:generators}
  At every iteration $k$ of \algref{manifoldsampling}, the finite set $\mathbb{Z}^k$ satisfies
  $F(x^k) \subseteq \mathbb{Z}^k \subset \cB( F(x^k); L_F\Delta_k )$
for $L_F$ as in \defref{constants}.
\end{assump}
Although \algref{manifoldsampling} does not assume access to $L_F$, ensuring that 
$\mathbb{Z}^k$ contains only elements of the form $F(y)$ for $y \in
\cB(x^k;\Delta_k)$ or $\alpha F(x^k) + (1-\alpha)(F(y))$ for $\alpha \in [0,1]$
will ensure that \assref{generators} is satisfied.

An ideal sample set $\mathbb{Z}^k$ would be one such that
$\Act{\mathbb{Z}^k} = \bigcup_{y\in\cB(x^k;\Delta_k)} \mathbb{A}(F(y))$;
that is, $\mathbb{Z}^k$ would contain points in $\mathbb{R}^p$ so that all selection
functions that define $h$ in the image of $\cB(x^k;\Delta_k)$ under $F$ are identified.
Fortunately, identifying all active selection functions near $x^k$ is not
necessary; this is indeed fortunate because ensuring that all such active selection functions 
have been identified may be impossible in practice. 
In our implementation, $\mathbb{Z}^k$ is initialized in either of the following 
ways, both of which 
are consistent with \assref{generators} and are practical:
 $\mathbb{Z}^k = \{  F(x^k) \}$ or  $\mathbb{Z}^k = \{ F(y)\colon y \in Y \subset \cB(x^k;\Delta_k)\}$.
We note that in the second case, additional evaluations of $F$ are not necessary; 
it is
sufficient to let $Y$ consist of points in $\cB(x^k;\Delta_k)$ where $F$ has
been evaluated during previous iterations of the algorithm.

\subsection{Smooth master model}\label{sec:master_model}
As previously stated, we want the gradient of the smooth master model to satisfy 
$g^k = \Proj(0,\cop{\Gen^k})$. This projection can be computed by solving the convex quadratic
optimization problem
\begin{equation}
   \label{eq:convex_qp}
    \begin{aligned}
    & \minimize_{\lambda} & & \lambda^\trans (G^k)^\trans G^k\lambda 
    & \text{subject to} & & e^\trans \lambda = 1, \; \lambda\geq 0. 
    \end{aligned}
\end{equation}
Employing $G^k$ and $D^k$ defined in \defref{useful_sets} and $\lambda^*$
solving~\eqref{eq:convex_qp}, we can define 
\begin{equation} \label{eq:gk_dk_def}
  g^k \defined G^k\lambda^* \qquad \mbox{ and } \qquad d^k \defined D^k\lambda^*.
\end{equation}
While there may not be a unique $\lambda^*$ solving~\eqref{eq:convex_qp},
the values of $g^k$ and $d^k$ in~\eqref{eq:gk_dk_def} are necessarily unique.

We employ $[d^k]_i$, the $i$th element of $d^k$, as weights attached to the $p$ smooth component models
$m^{F_i}_k$ to yield the smooth \emph{master model}
\begin{equation}
 m^f_k(x) \defined \sum_{i=1}^p [d^k]_i m^{F_i}(x).
 \label{eq:mastermodel}
\end{equation}
By construction, $\nabla m^f_k(x^k) = \ds\sum_{i=1}^p [d^k]_i \nabla m^{F_i}(x^k) =\nabla M(x^k)D^k\lambda^* 
= G^k\lambda^* = g^k.$
We draw attention to the fact that the master model $m^{f}_k$ is \emph{not} assumed to be an accurate model of $f$
in a Taylor approximation sense. 
However, \lemref{weird_v_approx} demonstrates a result resembling one direction
of the definition of ``order-1 subgradient accuracy,'' as in~\cite[Section 19.4]{AudetHare2020}.

\subsection{Sufficient decrease condition}\label{sec:suff_dec}
In iteration $k$ of the manifold sampling algorithm,
 the master model $m^f_k$ in~\eqref{eq:mastermodel} will be employed in the trust-region
subproblem 
\begin{equation}
  \minimize_{s\in \cB(0;\Delta_k)}\;  m^f_k(x^k+s).
 \label{eq:trsp}
\end{equation}
As with traditional trust-region methods, the problem~\eqref{eq:trsp} does not have to be
solved exactly. Rather, an approximate solution $s^k$
of~\eqref{eq:trsp} can be used, provided it satisfies a sufficient decrease
condition quantified by an algorithmic parameter $\kappad \in (0,1)$, namely,
\begin{equation} \label{eq:fraction_cauchy_decrease}
  \left\langle M(x^k) - M(x^k+s^k), d^k \right\rangle \ge
  \frac{\kappad}{2}\|g^k\|\min\left\{\Delta_k,\frac{\|g^k\|}{L_h\kappabmh}\right\}.
\end{equation}
(If $h$ is constant on $\cLmax$ or all of the models $m^{F_i}$ are linear and $L_h\kappabmh = 0$,
$\frac{\|g^k\|}{L_h\kappabmh}\defined \infty $.)

Note that the sufficient decrease condition~\eqref{eq:fraction_cauchy_decrease}
differs from traditional conditions employed in classical trust-region methods
(see, e.g.,~\cite[Theorem~6.3.3]{Trmbook}). Instead of measuring the decrease
in $m^f_k$ between $x^k$ and $x^k+s^k$, the left-hand side of~\eqref{eq:fraction_cauchy_decrease}
measures the decrease in terms of the specific convex combination, defined by $d^k$, of the gradients of 
selection functions at points near $F(x^k)$. The sufficient decrease
condition~\eqref{eq:fraction_cauchy_decrease} extends the approach
from~\cite{LMW16}, where 
$h = \left\| \cdot \right\|_1$ and decrease is measured by using the sign
pattern of $F$ at $x^k + s^k$.
We will demonstrate in \lemref{approx_solve} that an $s^k$
satisfying~\eqref{eq:fraction_cauchy_decrease} always exists and can be found
with a particular step choice. 

\begin{remark}[on the various uses of $\Delta_k$]
Note that the trust-region subproblem \eqref{eq:trsp},
\assref{flmodels} on model quality, and \assref{generators} 
on allowable sample sets
all involve the parameter $\Delta_k$. 
This intentional conflation of the use of $\Delta_k$ greatly facilitates our analysis. 
In analyses of derivative-free model-based methods, one commonly sees $\Delta_k$ controlling both trust-region radii
and model accuracy.
In manifold sampling, $\Delta_k$ plays a third role of the $\epsilon$ parameter
in the approximation of the Clarke $\epsilon$-subdifferential induced by $\Gen^k$. 
Whereas practical implementations of derivative-free model-based methods sometimes decouple $\Delta_k$ into
separate parameters controlling step sizes and model accuracy, one could also consider a third decoupling 
of $\Delta_k$ from its use in \assref{generators} at the expense of an algorithm that is more difficult to analyze. 
\end{remark}

\subsection{Manifold sampling loop}\label{sec:loop}
In the manifold sampling algorithm, 
the trial step $s^k$ suggested by the trust-region subproblem may not yield
sufficient decrease if $\Act{\mathbb{Z}^k}$ is a poor sample of nearby
active manifolds.
Therefore, after 
$s^k$ has been computed and $F(x^k+s^k)$ has been evaluated, but before the ratio determining success, $\rho_k$ (defined in~\eqref{eq:rho_ht}), is computed, 
$\mathbb{Z}^k$ sometimes  must be augmented,
resulting in a new master model and a new $s^k$.
We refer to this process that occurs in each iteration as the \emph{manifold sampling loop}. 

Although adding indices to $\mathbb{Z}^k$ 
may result in a given manifold sampling iteration requiring the solution of more than one
trust-region subproblem---and, more importantly, more than one evaluation of $F$ per
iteration---in practice the number of function evaluations per iteration is
rarely more than one.
We further remark that, even in theory, this manifold sampling loop cannot cycle indefinitely because the number
of selection functions defining $h$ is finite.

The termination of the manifold sampling loop hinges on a search for a sample
point $z \in \cop{ \{F(x^k), F(x^k+s^k)\}}$ and an index $j$ satisfying 
\begin{subequations}\label{eq:baoyu_condition}
 \begin{gather}
 j\in\mathbb{A}(z), \\
 \nabla h_j(z)^\trans (F(x^k)-F(x^k+s^{k})) \leq h(F(x^k)) - h(F(x^k + s^{k})).\label{eq:baoyu_decrease} 
 \end{gather} 
\end{subequations}
In other words, a point $z\in\R^p$ and an index $j$ in $\Act{z}$ satisfy \eqref{eq:baoyu_condition},
provided the affine function $h(F(x^k)) + \nabla h_j(z)^\top(F(x^k+s^k)-F(x^k))$ overestimates
$h(F(x^k+s^k))$. 
We remark that~\eqref{eq:baoyu_condition} is less stringent than the condition
in~\cite{KLW18}, which sought a manifold $j$ satisfying 
\begin{equation}\label{eq:jeff_condition}
h_j(F(x^k)) \le h(F(x^k)) 
\qquad\text{and}\qquad
h_j(F(x^k + s^k)) \ge h(F(x^k + s^k)).
\end{equation}
If $h$ is a continuous selection of affine functions,
which is the special case considered in~\cite{KLW18},
then \eqref{eq:jeff_condition}
is equivalent to fixing one affine function $h_j\in\Hdef$ that underestimates $h$ at $F(x^k)$
but overestimates $h$ at $F(x^k+s^k)$. 
Between~\cite{KLW18} and our present work, we see that the commonality lies in the overestimation
of $h$ at $F(x^k+s^k)$. 

We will demonstrate (in \lemref{existence_linearization}) that under our assumptions,
we can compute a $(z,j)$ pair satisfying \eqref{eq:baoyu_condition} for any $s^k$.
As our analysis will reveal, however, we additionally require the existence of a secondary $z'\in\mathbb{Z}^k$
satisfying 
\begin{subequations}\label{eq:matt_condition}
\begin{gather}
\ j\in\Act{z'}, \\
(s^k)^\trans  \nabla M(x^k)(\nabla h_j (z') - d^k) \leq 0 \label{eq:obtuse_property}
\end{gather}
\end{subequations}
for the same $j$ employed in \eqref{eq:baoyu_condition}. 
Geometrically, \eqref{eq:obtuse_property} requires that the trial step
$s^k$ be obtuse with the vector pointing to the single generator $\nabla M(x^k)\nabla h_j(z')$ from
the minimum norm of the convex hull of the generators, $\nabla M(x^k)d^k = g^k$. 
From the classical projection theorem (see, e.g.,~\cite[Theorem
  2.39]{rockafellar2009variational}),
 a trial step $s^k$ parallel to the steepest descent direction $-g^k$ satisfies \eqref{eq:obtuse_property}.
 Thus, if a trial step $s^k$ obtained from the solution of \eqref{eq:trsp} 
 fails to lead to the simultaneously satisfaction of 
 \eqref{eq:baoyu_condition} and \eqref{eq:matt_condition}, then we may safely replace
 $s^k$ with a default scaled steepest descent direction.

\subsection{$\rho_k$ test}
In common with classical trust-region methods,
manifold sampling employs a ratio test 
as a merit criterion. 
Whereas the value of $\rho_k$ in a classical trust-region method measures the ratio of
the actual decrease $f(x^k) - f(x^k + s^k)$ to predicted model decrease $m_k^f(x^k) - m_k^f(x^k+s^k)$,
the $\rho_k$ used in the manifold sampling algorithm is the
ratio of actual decrease in $F$ to predicted decrease in $M$, as weighted by the convex combination of gradients $d^k$. 
Explicitly, given $d^k$ defined in~\eqref{eq:gk_dk_def} 
and $s^k$ satisfying~\eqref{eq:fraction_cauchy_decrease},
$\rho_k$ is the ratio
\begin{equation} \label{eq:rho_ht}
  \rho_k \defined \ds\frac{\langle F(x^k) - F(x^k+s^k), d^k \rangle}
  {\langle M(x^k) - M(x^k+s^k), d^k \rangle}.
\end{equation}
A trial step $x^k + s^k$ is accepted only if $\rho_k$ is sufficiently
large.

\subsection{Algorithm statement}\label{sec:alg_state}

Having introduced the various algorithmic components, we can now 
state the algorithm along with restrictions
on algorithmic parameters in \algref{manifoldsampling}. 

\begin{algorithm2e}[t]
  \fontsize{8}{8}\selectfont
	\SetAlgoNlRelativeSize{-4}
	\SetKw{true}{true}
	\SetKw{break}{break}
	\SetKw{return}{return}
  \SetKwFunction{proc}{generate\_sk\_and\_j\_and\_z}
  \let\oldnl\nl
  \newcommand{\nonl}{\renewcommand{\nl}{\let\nl\oldnl}}

	Set $\eta_1\in(0,1)$, $\kappad \in (0,1)$, $\kappabmh \ge 0$, $\eta_2 \in (0,\etamax)$, $0 < \gammad < 1 \leq \gammai$, and $\Delta_{\max} > 0$ \label{line:initial_setting}
	
	Choose $\Delta_0$ satisfying $\Deltamax \ge \Delta_0>0$ and initial iterate $x^0$
	
	\For{$k=0,1,2,\ldots$  \label{line:outer_for_loop}}  
	{
		
    Evaluate $F$ as needed to build $p$ models $m^{F_i}_k$ satisfying \assref{flmodels} \label{line:build_models}

    Initialize $\mathbb{Z}^k$ satisfying \assref{generators}; form $D^k$ by \defref{useful_sets} \label{line:starting_gen_set}
    
		\While(\tcp*[f]{manifold sampling loop}){\true}{\label{line:outer_while} 
      Form $\nabla M(x^k)$ using $\nabla m^{F_i}_k(x^k)$; set $G^k \gets \nabla M(x^k)D^k$
      
      Solve~\eqref{eq:convex_qp} for $\lambda^{*}$; set $d^k \gets D^k\lambda^{*}$ \label{line:exact_solve_for_projection}

			Build master model $m^f_k$ using~\eqref{eq:mastermodel} with $\nabla m^f_k(x^k) = G^k\lambda^{*} = g^k$

			\eIf{$\Delta_k < \eta_2\| g^k\|$ \label{line:acceptabilityCheck}}
			{
        \proc

        \eIf{$j\in\Act{\mathbb{Z}^k}$\label{line:j_in_Act_test}}{
					Calculate $\rho_k$ using~\eqref{eq:rho_ht} and \break \tcp*[f]{acceptable iter.} \label{line:potential_acceptable} 
				}
				{
					Update $m^{F_i}_k$ (evaluating $F$ if needed) satisfying \assref{flmodels} \label{line:update_models}
					
          $\mathbb{Z}^k \gets \mathbb{Z}^k \cup \{z\}$ [$\cup \{F(y)\colon y \in \cB(x^k; \Delta_k)\}$]; form $D^k$ by \defref{useful_sets} \label{line:grow_Zk}
				}
			}
			{$\rho_k\gets 0;$ \break \tcp*[f]{unacceptable iter.}}
		}
		\eIf {
			$\rho_k>\eta_1>0$ \label{line:successCheck}}
		{
			$x^{k+1}\gets x^k+s^k$; $\Delta_{k+1}\gets\min\{\gammai\Delta_k,\Delta_{\max}\}$\label{line:grow_delta}
      \tcp*[f]{successful iter.}

		}
		{
			$x^{k+1}\gets x^k$; $\Delta_{k+1}\gets\gammad\Delta_k$ \tcp*[f]{unsuccessful iter.}

		}
	}
\SetNlSty{textbf}{P.}{}
  \setcounter{AlgoLine}{0}
  \SetKwProg{myproc}{Procedure}{}{}
  \nonl \myproc{\proc}{

				 Approximately solve~\eqref{eq:trsp} to obtain $s^{k}$ satisfying~\eqref{eq:fraction_cauchy_decrease} \label{line:approx_solve}
					
					Evaluate $F(x^{k} + s^{k})$   \label{line:evaluate_potential_update}
					
					Find $z \in \cop{\{F(x^k),F(x^k + s^k)\}}$ and $j$ satisfying~\eqref{eq:baoyu_condition} \label{line:linearization}
					
					\If{$j\in\Act{\mathbb{Z}^k}$, $\nexists z' \in \mathbb{Z}^k$ satisfying \eqref{eq:matt_condition}}
						{Set $s^k$ following \lemref{approx_solve} \\
						Evaluate $F(x^k+s^k)$; 
						find $z \in \cop{\{F(x^k),F(x^k + s^k)\}}$ and $j$ satisfying~\eqref{eq:baoyu_condition}
						}
          
          \return $s^k$, $j$, and $z$
  }
	\caption{Manifold sampling for general compositions (MSG) \label{alg:manifoldsampling}}
\end{algorithm2e}

We draw special attention to the following aspects of \algref{manifoldsampling} concerning subproblems.
\begin{description}[leftmargin=0.2in,labelindent=0in]
\item[\linerefa{exact_solve_for_projection}:] 
    The problem~\eqref{eq:convex_qp} can be solved \emph{exactly} in finite time, for example, by using 
    the specialized active-set algorithm of~\cite{kiwiel1986method}.
    Note that the direction $d^k$ produced in this line is used in
    \eqref{eq:fraction_cauchy_decrease} and \eqref{eq:obtuse_property}.
    
  \item[\linerefb{approx_solve}:] The existence of such an $s^{k}$ is guaranteed
    by \lemref{approx_solve}. The proof of \lemref{approx_solve} also
    provides an explicit construction for such an $s^k$. 
    Thus, even if a standard trust-region method applied to \eqref{eq:trsp} fails to 
    return a trial step $s^k$ satisfying \eqref{eq:fraction_cauchy_decrease}, 
   we can appeal to the construction in \lemref{approx_solve}. 

  \item[\linerefb{linearization}:] We again stress that the existence of such a $z$ is guaranteed
    by \lemref{existence_linearization}. 
    Moreover, in \lemref{candidate_search_lemma}, we will demonstrate that such a $z$ can be found
    in finite time without performing additional evaluations of $F$. 
    If multiple $j\in\mathbb{A}(z)$ are
    identified for the given $z$, then we arbitrarily select an
    element in \\
    $
    \argmax_{j\in\mathbb{A}(z)} \left\{ \nabla h_j(z)^\trans  \left( F(x^k + s^k) - F(x^k) \right)\colon j \text{ and }z\text{ satisfy }\eqref{eq:baoyu_condition}\right\}.
    $
\end{description} 
We additionally draw attention to other important aspects of \algref{manifoldsampling}. 
\begin{description}[leftmargin=0.2in,labelindent=0in]
  \item[\linerefa{initial_setting}:]   We note that $\kappabmh$ is
    an algorithmic parameter that bounds the norms of the Hessians of the
    models $m^{F_i}$.
  For analysis, we assume that $\etamax\in\R\cup\{\infty\}$, the upper
    bound on $\eta_2$, satisfies 
    \begin{equation}\label{eq:eta_max}
    \etamax \defined \min\left\{ \frac{1}{L_h\kappabmh}, \frac{\eta_1\kappad}{4\bigC} \right\}.
    \end{equation}
    (Again, for ease, define either $\frac{1}{L_h\kappabmh}$ or
    $\frac{\eta_1\kappad}{4\bigC}$ to be infinite in the pathological case when
    $L_h\kappabmh = 0$ or $C = 0$.)
    If $L_h\kappabmh$ is large, then~\eqref{eq:eta_max} may allow iterations to be deemed
    acceptable only when $\Delta_k$ is relatively small.
    Furthermore, as~\eqref{eq:eta_max} contains constants that
    are generally unknown, 
    \secref{implementation_details} discusses safeguards that can be included
    in a numerical implementation of \algref{manifoldsampling} if $\eta_2 > \etamax$.
    
  \item[{\bf \linerefa{outer_while}:}] \algref{manifoldsampling} will break out
    of the manifold sampling loop after at most $\left| \Hdef \right|-\left| 
    \mathbb{A}(F(x^k)) \right|$ times through.  The reason is that 
    $\left\{ j\colon h_j \in \Hdef \right\} \supseteq \mathbb{A}(\mathbb{Z}^k) \supseteq \mathbb{A}(F(x^k)) $
    and each time \linerefa{grow_Zk} is visited, the cardinality of $\mathbb{A}(\mathbb{Z}^k)$ will
    be increased by at least one. In the worst case, the indices of all
    selection functions in $\Hdef$ must be added to $\mathbb{Z}^k$ before
    $\rho_k$ can be calculated.    
    
   \item[{\bf \linerefa{update_models}:}] Depending on the means of
   model building being employed, one \emph{may} want to incorporate the function values 
   $F(x^k+s^k)$ computed in this inner iteration into the models $m_k^{F_i}$. 
   So long as the updated models $m_k^{F_i}$ satisfy \assref{flmodels}, as stated in this line
   of the algorithm, this will not affect convergence.
    
    \item[\linerefa{grow_delta}:] By \eqref{eq:nonincreasing_sequence_C_nonzero}, \eqref{eq:nonincreasing_sequence_C_zero}, and the
      updating of $\Delta_{k+1}$ on successful iterations, all points evaluated
      by \algref{manifoldsampling} are in the set $\cLmax$ as defined
      in~\eqref{eq:cLmaxdef}.\footnote{
      This claim assumes that any additional points evaluated during model construction in
      \linerefa{build_models} are also in $\cLmax$. 
      Allowing for points outside of $\cLmax$ is straightforward, provided the
      functions are defined wherever they are evaluated. 
      } 

     \item[Acceptable iterations:] Because any acceptable iteration satisfies
    \linerefa{acceptabilityCheck}, acceptable iterations occur when
    $\Delta_k<\eta_2\|\nabla m_k^f(x^k)\| = \eta_2 \| g^k \|$.
    That is, acceptable iterations occur when the norm of the master model
    gradient is sufficiently large relative to $\Delta_k$.
    On acceptable iterations,
    \begin{equation} \label{eq:result_of_params}
      \|g^k\|\ge
      \min\left\{ L_h\kappabmh\Delta_k, \left\| g^k \right\| \right\} \ge
      L_h\kappabmh\min\left\{ \Delta_k, \eta_2 \left\| g^k \right\| \right\} = L_h\kappabmh\Delta_k.
    \end{equation}
  \item[Successful iterations:] Successful iterations are acceptable
    iterations for which 
    $\rho_k>\eta_1$ and $x^{k+1}\gets x^k+s^k$. Note that on every
    successful iteration, 
    the gradient of the linearization, $d^k$, is represented in $\cop{\mathbb{D}^k}$ and 
     the decrease condition in~\eqref{eq:fraction_cauchy_decrease} is
        satisfied by $s^k$.

\end{description}

\section{Preliminary Analysis}\label{sec:prelims}
We now show preliminary results that will be used in the analysis of
\algref{manifoldsampling}.
We first show a result linking elements in 
$\cop{\Gen^k}$ to the subdifferentials of $f$ at nearby points. 

\begin{lemma}\label{lem:weird_v_approx}
  Let \assrefrange{f}{flmodels}
  hold, and let $x,y\in\cLmax$ satisfy
  $\|x-y\|\leq \Delta$.
  For any finite subsets $I$, $J$ and $I'$ such that $I \subseteq J \subseteq \{1,\ldots,\left| \Hdef \right|\}$  and $I'\subset\mathbb{N}$, define\\
    $\Gen \defined \{ \nabla M(x) \nabla h_i(z_{i'}) \colon i \in I, z_{i'}\in\cB(F(x);L_F\Delta), i'\in I' \}$ and\\
    $\cH \defined \cop{\{ \nabla F(y) \nabla h_j(F(y)) \colon j \in J\}}$.
    Then for each $g\in\cop{\Gen}$, there exists $v(g)\in \cH$
  satisfying
  \begin{equation}
    \label{eq:gapprox}
    \left\|g-v(g)\right\|\leq c_2 \Delta, 
  \end{equation}
  where $c_2$ is defined by 
  \begin{equation}\label{eq:c_2}
    c_2 \defined \Lh (\LnF + \kappag) + 2L_F^2L_{\nabla h_i}
  \end{equation}
  for $\Lh$, $\LnF$, and $\kappag$ from \defref{constants}.
\end{lemma}

\begin{proof}[Proof (adapted from {\cite[Lemma 4.1]{KLW18}})]
  Any $g\in\cop{\Gen}$ may be 
expressed as
  \begin{equation}
    \label{eq:weird_v_approx_g}
    g =  \ds \sum_{(i,i')\in I\times I'} \lambda_{i,i'} \nabla M(x) \nabla h_i(z_{i'}),
  \end{equation}
  where $z_{i'}\in\cB(F(x),L_F\Delta)$, $\sum_{(i,i')\in I \times I'} \lambda_{i,i'}=1$ and $\lambda_{i,i'} \geq 0$ for each
  $(i,i')\in I\times I'$.
 
  By supposition, $ \nabla F(y) \nabla h_i(F(y)) \in \cH$ for all $i \in I$.
  For  \\
  $
  v(g) \defined \ds\sum_{(i,i')\in I \times I'} \lambda_{i,i'} \nabla F(y) \nabla h_i(F(y)),
  $
  using the same $\lambda_{i,i'}$ as in~\eqref{eq:weird_v_approx_g} for $(i,i')\in I\times I'$, convexity of $\cH$ implies that $v(g)\in \cH$. 
  Since $y\in\cB(x;\Delta)$ and using \subassref{f}{Fval},
  \subassref{f}{Fgrad}, \subassref{h}{func}, \subassref{h}{grad}, 
  and \assref{flmodels}, we have
  \begin{align*}
  \left\| \nabla M(x)\nabla h_i(z_{i'}) - \nabla F(y) \nabla h_i(F(y)) \right\| 
  \le&\; (\Lh \LnF + 2L_F^2L_{\nabla h_i} + \kappag \Lh) \Delta
  \end{align*}
  for each $(i,i')$. 
   The
  definition of $v(g)$ and~\eqref{eq:weird_v_approx_g} then imply
  \begin{align*}
    \left\|g-v(g)\right\| 
    &\leq \left\|  \ds\sum_{(i,i')\in I\times I'} \left[ \lambda_{i,i'} \nabla 
M(x) \nabla h_i(F(x)) - \lambda_{i,i'} \nabla F(y) \nabla h_i(F(y)) \right] \right\|\\
    &\le  \ds\sum_{(i,i')\in I\times I'}\lambda_{i,i'} \left\|\nabla M(x) \nabla h_i(F(x)) - \nabla 
F(y) \nabla h_i(F(y)) \right\|
\le c_2 \Delta.
  \end{align*}
\end{proof}

For simplicity, in the rest of the paper we drop the superscripts of $x^k$ and $s^k$ when possible.
The next lemma guarantees that the condition in \linerefb{linearization} of
\algref{manifoldsampling} can always be attained.

\begin{lemma}\label{lem:existence_linearization}
  If \assref{h} holds, then there exist $z\in\cop{\{F(x),
  F(x+s)\}}$ and index $j$ satisfying \eqref{eq:baoyu_condition}.
\end{lemma}

\begin{proof}
  We first define
  \begin{equation}\label{eq:zalpha}
    z(\alpha) \defined \alpha F(x) + (1 - \alpha) F(x+s)
  \end{equation}
  and will show there exists an $\alpha\in [0,1]$ such that $z(\alpha)$ and 
  $j \in \Act{z(\alpha)}$ satisfy~\eqref{eq:baoyu_condition}.
	
  We prove by contradiction. Suppose there exists no such $\alpha$. When \subassref{h}{h} and \subassref{h}{Ieh} hold, by
  \lemref{directional_derivative_inf}, 
  \begin{equation*}
    \begin{aligned}
      h(F(x)) - h(F(x+s)) \geq &\int_{0}^{1} \inf_{j\in\mathbb{A}(z(\alpha))}\{ \nabla h_j(z(\alpha))^\trans (F(x) - F(x+s)) \} d\alpha \\
      > &\int_{0}^{1} (h(F(x)) - h(F(x+s))) d\alpha = h(F(x)) - h(F(x+s)),
    \end{aligned}
  \end{equation*}
	which is a contradiction. Therefore, the result is shown.
\end{proof}

\begin{algorithm2e}[t]
  \fontsize{8}{8}\selectfont
	\DontPrintSemicolon 
	\SetAlgoNlRelativeSize{-5}
	\SetKw{true}{true}
	\SetKw{break}{break}
	\lIf{$F(x)$ and some $j$ satisfy~\eqref{eq:baoyu_condition}}{return $F(x)$ and $j$} 
	\lIf{$F(x+s)$ and some $j$ satisfy~\eqref{eq:baoyu_condition}}{return $F(x+s)$ and $j$} 

	\For{$l = 1,2,\ldots$}
  {Generate $2^{l-1}$ candidates 
  	$\left\{\frac{2k-1}{2^l}\colon k=1,\ldots,2^{l-1}\right\}$
		\\
		\For{$k=1,2,\ldots,2^{l-1}$}
    {$\alpha \leftarrow \frac{2k-1}{2^l}$ and set $z(\alpha)$ as in~\eqref{eq:zalpha} \\
		\lIf{$z(\alpha)$ and some $j$ satisfy~\eqref{eq:baoyu_condition}}{return $z(\alpha)$ and $j$}}
	}
	\caption{Grid Search for $z$ and $j$ \label{alg:gridsearch}}
\end{algorithm2e}

Having demonstrated an existence result in \lemref{existence_linearization},
we now show that \algref{gridsearch} produces 
a point $z$ and index $j$ satisfying \eqref{eq:baoyu_condition} in
finite time.
In numerical experiments, we use the bisection search of 
\algref{bisecsearch} (in Appendix~\ref{sec:appendix}) instead of \algref{gridsearch}.
While we cannot show that \algref{bisecsearch} terminates in finite time, in
practice it is faster, and we have yet to encounter issues with 
termination.\footnote{One can construct more efficient approaches for
identifying $z$ and $j$ when $h$ takes specific forms, but we present
\algref{gridsearch} and \algref{bisecsearch} because of their general applicability
to functions satisfying \assref{h}.}

\begin{lemma}\label{lem:candidate_search_lemma}
  If \assref{h} holds, then \algref{gridsearch} returns a point $z$ and
  index $j$ satisfying \eqref{eq:baoyu_condition} in finitely many iterations. 
\end{lemma}
\begin{proof}
	If $F(x) = F(x+s)$, then the proof is trivial. Therefore, consider $F(x) \neq F(x+s)$.
	Let \subassref{h}{h} and \subassref{h}{Ieh} hold.
	Suppose that for all $j\in\mathbb{A}(F(x))$, 
	\begin{equation}
	\label{eq:stub1}
	\nabla h_j(F(x))^\trans (F(x)-F(x+s)) > h(F(x)) - h(F(x+s)),
	\end{equation}
	or else it is clear that \algref{gridsearch} terminates at the very first line.
	
  Recall the definition of $z(\alpha)$ in~\eqref{eq:zalpha}.
	By \defref{pc1manifold}, there exist index $j_1$ and $\alpha_{j_1} < 1$ such that $j_1\in\mathbb{A}(F(x))$ 
  and also $j_1\in\mathbb{A}(z(\alpha))$ for all 
	$\alpha\in[0,\alpha_{j_1}]$.  
  Because $h_{j_1}$ is continuously differentiable, $\nabla
  h_{j_1}(z(\alpha))^\trans (F(x) - F(x+s))$ is a continuous function of $\alpha$ for
  $\alpha\in [0,1]$. Combined with~\eqref{eq:stub1}, there must exist $\bar{\alpha}_{j_1} \in
  (0,\alpha_{j_1}]$ such that
  $
    \nabla h_j(z(\alpha))^\trans (F(x)-F(x+s)) > h(F(x)) - h(F(x+s))
    $
	for all $\alpha\in[0,\bar{\alpha}_{j_1}]$.
	
	We now show by contradiction that there must exist some $\alpha \in [\bar{\alpha}_{j_1},1)$ such that for some $j\in\mathbb{A}(z(\alpha))$,
  $
	\nabla h_j(z(\alpha))^\trans (F(x) - F(x+s)) < h(F(x)) - h(F(x+s)).
  $
	Suppose such an $\alpha$ does not exist. By \lemref{directional_derivative_inf}, 
	\begin{equation*}
	\begin{aligned}
	h(F(x)) - h(F(x+s)) \geq &\int_0^1 \inf_{j\in\mathbb{A}(z(\alpha))}\{ \nabla h_j(z(\alpha))^\trans (F(x) - F(x+s)) \} d\alpha \\
	= &\int_0^{\bar{\alpha}_{j_1}} \inf_{j\in\mathbb{A}(z(\alpha))}\{ \nabla h_j(z(\alpha))^\trans (F(x) - F(x+s)) \} d\alpha \\
	&+ \int_{\bar{\alpha}_{j_1}}^{1} \inf_{j\in\mathbb{A}(z(\alpha))}\{ \nabla h_j(z(\alpha))^\trans (F(x) - F(x+s)) \} d\alpha \\
	> &\int_0^1 (h(F(x)) - h(F(x+s))) d\alpha = h(F(x)) - h(F(x+s)),
	\end{aligned}
	\end{equation*}
	which is a contradiction.  
	Thus, there exist $\alpha^*\in [\bar{\alpha}_{j_1},1)$ and $j^*\in\mathbb{A}(z(\alpha^*))$ satisfying\\
  $
	\nabla h_{j^*}(z(\alpha^*))^\trans (F(x) - F(x+s)) < h(F(x)) - h(F(x+s)).
  $
  Using the same arguments as previously, we have from
  \defref{pc1manifold} that there exists $\epsilon_{\alpha^*} > 0$ such that
  at least one of
  $j^*\in\mathbb{A}(z(\alpha))$  for all $\alpha\in (\alpha^* -
  \epsilon_{\alpha^*}, \alpha^*)$ or
   $j^*\in\mathbb{A}(z(\alpha))$  for all $\alpha\in
  (\alpha^*,\alpha^*+\epsilon_{\alpha^*})$ 
  holds. 
  Without loss of generality, suppose $j^*\in\mathbb{A}(z(\alpha))$ for all $\alpha\in(\alpha^*-\epsilon_{\alpha^*},\alpha^*)$.
   By the continuity and
  smoothness of $h_{j^*}$ (\subassref{h}{grad}), there exists $\bar{\epsilon}_{\alpha^*} \leq
  \epsilon_{\alpha^*}$ such that for all $\alpha\in (\alpha^* -
  \bar{\epsilon}_{\alpha^*}, \alpha^*)$, we have
  $
	\nabla h_{j^*}(z(\alpha))^\trans (F(x) - F(x+s)) < h(F(x)) - h(F(x+s)).
  $
  \algref{gridsearch} will evaluate a point within the
  interval $(\alpha^* - \bar{\epsilon}_{\alpha^*}, \alpha^*)$ within at most $\lceil\log_{\frac{1}{2}}
  \bar{\epsilon}_{\alpha^*}\rceil + 1$ iterations.
    Hence, \algref{gridsearch}
  must terminate in finite time.
\end{proof}

We now demonstrate that \linerefb{approx_solve} in
\algref{manifoldsampling} is always satisfiable,
even if the trust-region subproblem solver does not identify such a solution. 
We note that \linerefb{approx_solve} is not reached if $0\in \cop{\Gen^k}$ 
by virtue of the acceptability criterion.

\begin{lemma} \label{lem:approx_solve}
  Let $d^k = D^k\lambda^{*}$ be obtained from
  {\rm \linerefa{exact_solve_for_projection}} 
  of \algref{manifoldsampling}, and therefore $g^k = \nabla M(x^k) d^k $.
   If
   \subassref{h}{func} and \assref{flmodels} are satisfied and \eqref{eq:eta_max} holds, 
   then $\hat{s} \defined - \Delta_k \frac{g^k}{\left\| g^k \right\|}$
satisfies \eqref{eq:fraction_cauchy_decrease} (in place of $s^k$). 
\end{lemma}

\begin{proof}
  From \subassref{h}{func}, \defref{constants}, \defref{useful_sets}, \eqref{eq:convex_qp}, and \eqref{eq:gk_dk_def}, we know that
	\begin{equation}\label{eq:normd}
		\|d^k\| = \|D^k\lambda^*\| \le \max_{d\in\mathbb{D}^k} \|d\| \le L_h.
	\end{equation}
	Moreover, from \assref{flmodels} and \defref{constants}, for any $s\in\mathbb{R}^n$ we have
	\begin{equation}\label{eq:norm_diff}
	\begin{aligned}
		&\|M(x^k) + \nabla M(x^k)^\trans s - M(x^k+s)\| \\
		 \le &\sum_{i=1}^p|m^{F_i}(x^k) + \nabla m^{F_i}(x^k)^\trans s - m^{F_i}(x^k + s)| 
		 \le \sum_{i=1}^p \frac{1}{2}\kappaibmh\|s\|^2 = \frac{1}{2}\kappabmh\|s\|^2,
	\end{aligned}
	\end{equation}
  where the last inequality comes from Taylor's theorem. Combining \eqref{eq:normd}, \eqref{eq:norm_diff}, and the definition of $\hat{s}$, we have
	\begin{equation*}
	\begin{aligned}
	&(d^k)^\trans(M(x^k) + \nabla M(x^k)^\trans \hat{s} - M(x^k + \hat{s})) \\
	\ge &-\|d^k\|\|M(x^k) + \nabla M(x^k)^\trans \hat{s} - M(x^k + \hat{s})\| 
	\ge -\frac{1}{2}L_h\kappabmh\|\hat{s}\|^2 = -\frac{1}{2}L_h\kappabmh\Delta_k^2.
	\end{aligned}
	\end{equation*}
	Using \eqref{eq:eta_max}, we get \eqref{eq:fraction_cauchy_decrease} by
	\begin{equation*}
	\begin{aligned}
	&\left\langle M(x^k) - M(x^k+\hat{s}), d^k \right\rangle \ge -(d^k)^\trans\nabla M(x^k)^\trans \hat{s} - \frac{1}{2}L_h\kappabmh\Delta_k^2 \\
	= &-(g^k)^\trans\hat{s} - \frac{1}{2}L_h\kappabmh\Delta_k^2 = \|g^k\|\Delta_k - \frac{1}{2}L_h\kappabmh\Delta_k^2 \\
	\ge &\|g^k\|\Delta_k - \frac{1}{2}\|g^k\|\Delta_k = \frac{1}{2}\|g^k\|\Delta_k 
	\ge \frac{\kappad}{2}\|g^k\|\min\left\{\Delta_k,\frac{\|g^k\|}{L_h\kappabmh}\right\}.
	\end{aligned}
	\end{equation*}
\end{proof}

\lemref{linearization_condition} is a technical result concerning decrease in the objective function that we will employ in \secref{analysis2}. 

\begin{lemma}\label{lem:linearization_condition}
 Let \assrefrange{f}{generators} hold. If iteration $k$ of \algref{manifoldsampling} is acceptable,
 then
  \begin{equation}\label{eq:ht_def}
      h(F(x^k)) - h(F(x^k+s^k)) \geq (d^k)^\trans (F(x^k)-F(x^k+s^k)) - \bigC\Delta_k^2
  \end{equation}
  for $\bigC$ as in \defref{constants}.
\end{lemma}

\begin{proof}
  Because iteration $k$ is acceptable, there is some point \linebreak[4] $z \in
  \cop{\{F(x^k),F(x^k + s^k)\}}$ and $j\in\mathbb{A}(\mathbb{Z}^k)$ such that
  \eqref{eq:baoyu_decrease} holds. 
  Furthermore, either some $z'\in\mathbb{Z}^k\subset \R^p$ with $j\in\Act{z'}$
  exists that satisfies \eqref{eq:obtuse_property}, or $s^k$ is generated
  following \lemref{approx_solve} with $\nabla M(x^k)\nabla
  h_j(z')\in\cop{\Gen^k}$.  We show that in either case
  \begin{equation}\label{eq:obtuse_projection}
  (\nabla h_j(z') - d^k)^\trans \nabla M(x^k)^\trans s^k \leq 0.
  \end{equation}
  
  When \eqref{eq:obtuse_property} is satisfied by $z'\in\mathbb{Z}^k$ and $j\in\Act{z'}$, \eqref{eq:obtuse_projection}
  follows immediately from \eqref{eq:obtuse_property}.
  When $s^k$ is generated from \lemref{approx_solve} with $\nabla M(x^k)\nabla h_j(z')\in\cop{\Gen^k}$, then 
  $s^k = -\Delta_k\frac{g^k}{\|g^k\|}$.
  By the classical projection theorem (see, e.g.,~\cite[Theorem 2.39]{rockafellar2009variational}),
  \begin{equation*}
  (\nabla h_j(z') - d^k)^\trans \nabla M(x^k)^\trans s^k = -(\nabla h_j(z') - d^k)^\trans \nabla M(x^k)^\trans \Delta_k\frac{g^k}{\|g^k\|} \leq 0,
  \end{equation*}
  which is exactly \eqref{eq:obtuse_projection}.

  From \subassref{f}{Fgrad},
  $
  |F_i(x^k + s^k) - F_i(x^k) - \nabla F_i(x^k)^\trans  s^k| \leq \frac{\LnFi}{2}\|s^k\|^2
  $
  for $i=1,\ldots,p$. By the definition of $\LnF$ in \defref{constants},
  \begin{equation}\label{eq:intermediate}
  \|F(x^k + s^k) - F(x^k) - \nabla F(x^k)^\trans  s^k\| \leq \frac{\LnF}{2}\|s^k\|^2.
  \end{equation}
  Therefore, by the Cauchy--Schwarz inequality,
  \begin{equation}\label{eq:int1}
  \begin{aligned}
  & (\nabla h_j(z) - d^k)^\trans (F(x^k + s^k) - F(x^k) - \nabla F(x^k)^\trans  s^k)\\
  &\leq \|\nabla h_j(z) - d^k\| \|F(x^k + s^k) - F(x^k) - \nabla F(x^k)^\trans  s^k\| \leq (2L_h) \frac{\LnF}{2}\|s^k\|^2 .
  \end{aligned}
  \end{equation}

  Thus,
  \begin{equation*}
      \begin{aligned}
      &h(F(x^k + s^{k})) - h(F(x^k)) - (d^k)^\trans (F(x^k + s^{k}) - F(x^k)) \\
      \leq &(\nabla h_j(z) - d^k)^\trans (F(x^k + s^{k}) - F(x^k)) \\
      \leq &(\nabla h_j(z) - d^k)^\trans \nabla F(x^k)^\trans s^{k} + L_h\LnF \|s^{k}\|^2 \\
      = &(\nabla h_j(z') - d^k)^\trans \nabla F(x^k)^\trans s^{k} + (\nabla h_j(z) - \nabla h_j(z'))^\trans \nabla F(x^k)^\trans s^{k} + L_h\LnF \|s^{k}\|^2 \\
      \leq &(\nabla h_j(z') - d^k)^\trans (\nabla M(x^k)^\trans s^{k} + (\nabla F(x^k) - \nabla M(x^k))^\trans s^{k}) \\
        & + 2L_{\nabla h}L_F^2\Delta_k^2 + L_h\LnF\Delta_k^2 \\
      = &(\nabla h_j(z') - d^k)^\trans \nabla M(x^k)^\trans s^k + (\nabla h_j(z') - d^k)^\trans (\nabla F(x^k) - \nabla M(x^k))^\trans s^{k}  \\
      &+ 2L_{\nabla h}L_F^2\Delta_k^2 + L_h\LnF\Delta_k^2  \\
      \leq &0 + 2L_h\kappag\Delta_k^2 + 2L_{\nabla h}L_F^2\Delta_k^2 + L_h\LnF\Delta_k^2 
      = \bigC \Delta_k^2,
      \end{aligned}
  \end{equation*}
  where the first inequality comes from~\eqref{eq:baoyu_decrease}; the second
  inequality comes from \eqref{eq:int1}; 
  the third
  inequality comes from \subassref{f}{Fval}, \subassref{h}{grad}, and \assref{generators}; and the
  last inequality comes from \subassref{h}{func}, \assref{flmodels}, and
  \eqref{eq:obtuse_projection}.
 \end{proof}

Between \lemref{approx_solve} and \lemref{linearization_condition}, we have established that
every acceptable iterate in \algref{manifoldsampling}
satisfies~\eqref{eq:fraction_cauchy_decrease} and~\eqref{eq:ht_def} simultaneously.

\section{Analysis of Manifold Sampling}\label{sec:analysis2}
We now show that cluster points of the sequence of iterates
generated by \algref{manifoldsampling} are Clarke stationary. The proof 
uses the following sequence of results.
\begin{description}[leftmargin=0.2in,labelindent=0in,labelsep=0.4em]
  \item[\lemref{success_when_linear}] shows that when the trust-region radius
    $\Delta_k$ is a sufficiently small multiple of $\left\| g^k \right\|$, the norm of the master model gradient,
    the iteration will be successful.
  \item[\lemref{delta_to_0_linear}] shows that $\ds\lim_{k\to\infty}\Delta_k=0$.
  \item[\lemref{g_to_0_linear}] shows that as $k\to\infty$, a subsequence of master model
    gradients $g^k$ must go to zero as well.
  \item[\lemref{g_to_v_linear}] shows that zero is in the generalized Clarke
    subdifferential $\partialC f(x^*)$ of any cluster point $x^*$ of any subsequence of iterates where
    the master model gradients go to zero.
  \item[\thmref{cluster_linear}] shows that $0 \in \partialC f(x^*)$ for any 
    cluster point $x^*$ of the
    sequence of iterates generated by \algref{manifoldsampling}.
\end{description}
We remark that the proofs of \lemref{g_to_0_linear},
\lemref{g_to_v_linear}, and \thmref{cluster_linear} are similar to analogous
results in~\cite{KLW18, LMW16}, but we have included them for completeness.

We first demonstrate that 
a successful iteration occurs whenever 
the trust-region radius 
is smaller than a constant multiple of 
the norm of the master model gradient.

\begin{lemma}\label{lem:success_when_linear}
  Let \assreftwo{h}{flmodels} hold. If an iteration is acceptable and 
  \begin{equation}\label{eq:delta_bound_linear}
    \Delta_k < \frac{\kappad(1-\eta_1)}{4 \kappaf \Lh} \|g^k\|
  \end{equation}
  (where the pathological case of $\kappaf = 0$ or $\Lh=0$ results in an
  infinite right-hand side),
  then $\rho_k > \eta_1$ in
  \algref{manifoldsampling}, and the iteration is successful.
\end{lemma}

\begin{proof}
 Because the iteration is acceptable, $g^k \neq 0$, and so the
 right-hand side of~\eqref{eq:delta_bound_linear}  is positive. 
  Using the definition of $\rho_k$ in~\eqref{eq:rho_ht},
we have
  \begin{align}\label{eq:num_bound}
      1-\rho_k\leq & \left| \rho_k - 1 \right| \nonumber \\
      = &\left| \frac{\langle F(x^k) - F(x^k+s^k), d^k \rangle} 
      { \langle M(x^k) - M(x^k+s^k), d^k \rangle} - 1\right|\nonumber \\
     = & \displaystyle\frac{\left| \left\langle F(x^k) - F(x^k + s^k), d^k \right\rangle - \left\langle M(x^k) - M(x^k+s^k), d^k \right\rangle \right|}{\langle M(x^k) - M(x^k+s^k), d^k \rangle} \nonumber\\
      \le&\, \displaystyle\frac{\left\| F(x^k) - M(x^k) \right\| \left\| d^k \right\| + \left\| F(x^k+s^k) - M(x^k+s^k) \right\| \left\| d^k \right\|}{\langle M(x^k) - M(x^k+s^k), d^k \rangle}\nonumber\\
      \le&\, \displaystyle\frac{2 \kappaf \Lh \Delta_k^2}{\langle M(x^k) - M(x^k+s^k), d^k \rangle} \nonumber\\
      \le&\, \displaystyle\frac{4 \kappaf \Lh \Delta_k^2}{\kappad \left\| g^k \right\| \min \left\{
      \Delta_k, \frac{\left\| g^k \right\|}{L_h\kappabmh} \right\}} & \jtext{by~\eqref{eq:fraction_cauchy_decrease}} \nonumber\\
      = &\frac{4 \kappaf \Lh \Delta_k}{\kappad \left\| g^k \right\|} & \jtext{by~\eqref{eq:result_of_params}},
    \end{align}
    where the second inequality holds by \assref{flmodels}, \subassref{h}{func}, and the fact that $\|s^k\|\leq\Delta_k$ and $\|d^k\|\leq\Lh$.
  Applying~\eqref{eq:delta_bound_linear} to~\eqref{eq:num_bound} yields
  \[
  1-\rho_k \le \frac{4 \kappaf \Lh \Delta_k}{\kappad \left\| g^k
  \right\| } < 1-\eta_1.
  \]
  Thus, $\rho_k>\eta_1$ if $\Delta_k$ satisfies~\eqref{eq:delta_bound_linear}, 
  and the iteration is successful.
\end{proof}

We now show that the sequence of trust-region radii converges to zero.

\begin{lemma}\label{lem:delta_to_0_linear}
  Let \assrefrange{f}{generators} hold. If 
$\{x^k,\Delta_k\}_{k\in\mathbb{N}}$ is generated by
  \algref{manifoldsampling}, 
  then the sequence
  $\{f(x^k)\}_{k\in\mathbb{N}}$ is nonincreasing, and $\ds\lim_{k\to\infty}\Delta_k=0$.
\end{lemma}

\begin{proof}
  If iteration $k$ is unsuccessful, then $\Delta_{k+1}<\Delta_k$, and
  $x^{k+1}=x^k$; therefore, $f(x^{k+1})=f(x^k)$. On successful iterations $k$, by \assrefrange{f}{flmodels}, we know that
  \begin{equation}
  \begin{aligned}
    \label{eq:nonincreasing_sequence}
  	f(x^k) - f(x^{k+1}) &\ge (d^k)^\trans (F(x^k) - F(x^{k+1})) - \bigC\Delta_k^2 & \jtext{by~\eqref{eq:ht_def}} \\
  	&= \rho_k (d^k)^\trans (M(x^k) - M(x^{k+1})) - \bigC\Delta_k^2 & \jtext{by~\eqref{eq:rho_ht}}\\
  	&\ge \rho_{k}\frac{\kappad}{2}\|g^{k}\|\min\left\{\Delta_{k},\frac{\|g^{k}\|}{L_h\kappabmh}\right\} - \bigC\Delta_{k}^2 & \jtext{by~\eqref{eq:fraction_cauchy_decrease}} \\
  	&= \rho_{k}\frac{\kappad}{2}\|g^{k}\|\Delta_{k} - \bigC\Delta_{k}^2 > \eta_1\frac{\kappad}{2}\|g^{k}\|\Delta_{k} - \bigC\Delta_{k}^2 & \jtext{by~\eqref{eq:eta_max}}.
  \end{aligned}
  \end{equation}
  Based on \eqref{eq:nonincreasing_sequence}, if $C > 0$, we have 
  \begin{equation}
  \begin{aligned}
   \label{eq:nonincreasing_sequence_C_nonzero}
   f(x^k) - f(x^{k+1}) &> \eta_1\frac{\kappad}{2}\|g^{k}\|\Delta_{k} - \bigC\Delta_{k}^2 & \jtext{by~\eqref{eq:nonincreasing_sequence}} \\
   &> \eta_1\frac{\kappad}{2}\frac{4\bigC}{\eta_1\kappad}\Delta_{k}^2 - \bigC\Delta_{k}^2 = \bigC\Delta_{k}^2 > 0 & \jtext{by~\eqref{eq:eta_max}}.
  \end{aligned}
  \end{equation}
  On the other hand, if $C = 0$, by \eqref{eq:nonincreasing_sequence} we have
  \begin{equation}
  \label{eq:nonincreasing_sequence_C_zero}
  f(x^k) - f(x^{k+1}) > \eta_1\frac{\kappad}{2}\|g^{k}\|\Delta_{k} > \frac{\kappad\eta_1}{2\eta_2}\Delta_{k}^2 > 0.
  \end{equation}

  Thus, the sequence $\{f(x^k)\}_{k\in\mathbb{N}}$ is nonincreasing.  

  To show that $\Delta_k \to 0$, we separately consider the cases when there
  are infinitely or finitely many successful iterations.
  First, suppose that there are infinitely many successful iterations, indexed by $\left\{
  k_j \right\}_{j\in\mathbb{N}}$. Since $f(x^k)$ is nonincreasing in
  $k$ and $f$ is bounded below
  (by \subassref{f}{level}, \subassref{f}{Fval}, and \subassref{h}{func}), 
  the sequence $\{f(x^k)\}_{k\in\mathbb{N}}$ converges
  to some limit $f^*\leq f(x^0)$.  Thus, having infinitely
  many successful iterations (indexed $\{k_j\}_{j \in \mathbb{N}}$) implies that there exists a positive constant $\overline{C} >0$ such that
  \begin{equation}
  \begin{aligned}
    \label{eq:eta_delta_normg_linear}
      \infty
      &> f(x^0) - f^*
      \ge \ds\sum_{j=0}^\infty f(x^{k_j}) - f(x^{k_{j+1}}) > \ds\sum_{j=0}^\infty  \overline{C}\Delta_{k_j}^2
  \end{aligned}
  \end{equation}
  by \eqref{eq:nonincreasing_sequence_C_nonzero} and \eqref{eq:nonincreasing_sequence_C_zero}.
  It follows that $\Delta_{k_j} \to 0$ for the sequence of
  successful iterations. Observe that
  $\Delta_{k_j+1}\le\gammai\Delta_{k_j}$ and that
  $\Delta_{k+1}=\gammad\Delta_k<\Delta_k$ if iteration $k$ is
  unsuccessful.  Thus, for any unsuccessful iteration $k>k_j$,
  $\Delta_k\leq\gammai\Delta_q$, where $q\defined\max\{k_j\colon j\in\mathbb{N},\,k_j<k\}$. 
  It follows immediately that 
  $
  0\leq \lim_{k\to\infty}\Delta_k \leq
  \gammai\lim_{j\to\infty}\Delta_{k_j} = 0,
  $
and so $\Delta_k \to 0$ in this case. 

Next, suppose there are only finitely many successful iterations and
let $\nu\in\mathbb{N}$ be the number of successful iterations.
Since $\gammad<1\leq\gammai$, it follows that $0\leq\Delta_k\leq
\gammai^\nu\gammad^{k-\nu}\Delta_0$ for each $k\in\mathbb{N}$. Thus,
$\Delta_k \to 0$.
\end{proof}

We now show that the norms of the master model gradients are not bounded away
from zero.

\begin{lemma} \label{lem:g_to_0_linear}
  Let \assrefrange{f}{generators}  hold. If the sequence 
  $\{x^k,\Delta_k\}_{k\in\mathbb{N}}$ is generated by
  \algref{manifoldsampling}, then 
  $\ds\liminf_{k \to \infty} \| g^k \| = 0$. 
\end{lemma}

\begin{proof}
  To obtain a contradiction, suppose there is an iteration $j$ and some
  $\epsilon>0$ for which $\|g^k\|\geq\epsilon$ for all $k\geq j$.  
  \algref{manifoldsampling} guarantees that $\Delta_j\geq\gammad^j\Delta_0>0$.
  With Assumptions~\ref{ass:h} and \ref{ass:flmodels}, any
  iteration
  where $\Delta_k < V \left\| g^k \right\|$ for
  $
  V \defined \min\left\{ \eta_2, \frac{\kappad (1-\eta_1)}{4 \kappaf
      \Lh} \right\}
      $
  will be successful because the conditions of
  \lemref{success_when_linear} are then satisfied.
  Therefore, by the contradiction hypothesis, any $k\geq j$ satisfying 
  $\Delta_k < V \epsilon$
  is guaranteed to be successful, in which case $\Delta_{k+1} = \gammai \Delta_k
  \ge \Delta_k$.  On the other hand, if $\Delta_k\ge V\epsilon$,
  then $\Delta_{k+1}\geq\gammad \Delta_k$.
  Under \assrefrange{f}{flmodels}, a straightforward inductive argument then yields $\Delta_k \geq
  \min(\gammad V\epsilon,\Delta_j)>0$ for all $k\geq j$,
  contradicting \lemref{delta_to_0_linear}.  Thus, no such
  $(j,\epsilon)$ pair exists, and so $\ds \liminf_{k\to\infty}\|g^k\|=0$.
\end{proof}

The next lemma shows that subsequences of iterates with master model gradients
converging to $0$ have cluster points that are Clarke stationary.
\algref{manifoldsampling} generates at least one
such subsequence of iterates by \lemref{g_to_0_linear}.

\begin{lemma}
\label{lem:g_to_v_linear}
Let \assrefrange{f}{generators} hold, and let 
$\{x^k,\Delta_k,g^k\}_{k\in\mathbb{N}}$ be a sequence  
generated by \algref{manifoldsampling}.
For any subsequence $\{k_j\}_{j\in\mathbb{N}}$ of acceptable iterations such 
that both
\[
\lim_{j\to\infty}\|g^{k_j}\|=0
\]
and $\{x^{k_j}\}_{j\in\mathbb{N}} \to x^*$ for some cluster point $x^*$, then $0 \in \partialC
f(x^*)$.
\end{lemma}

\begin{proof}
By continuity of $F_i$ (\subassref{f}{Fval}), there exists
  $\bar\Delta > 0$ so that for all $\Delta\in [0,\bar\Delta]$, the manifolds
  active in $\cB(x^*;\Delta)$ are precisely the manifolds active at
  $x^{*}$; that is, 
  \begin{equation}
  \label{eq:activity_settles} 
  \mathbb{A}(F(x^{*})) = \bigcup_{y \in \cB(F(x^{*});L_F\Delta)} \mathbb{A}(y) \qquad 
    \mbox{ for all } \Delta \le \bar{\Delta}.
  \end{equation}
   Thus, because $\Delta_k\to 0$
   by \lemref{delta_to_0_linear} and because $\{x^{k_j}\}_{j\in\mathbb{N}}$
  converges to $x^*$ by supposition, 
  we may conclude that for $j$ sufficiently
  large, $\Act{\mathbb{Z}^{k_j}}\subseteq\Act{F(x^*)}$.
   By
  \lemref{weird_v_approx} with $I
  \gets \Act{\mathbb{Z}^{k_j}}$, 
    $J\gets \Act{F(x^*)}$, 
  $x \gets x^{k_j}$, $y \gets x^*$, and \linebreak[4] $\Delta\gets\max\left\{\Delta_{k_j},\left\|x^{k_j} - x^*\right\|\right\}$,
  there exists $v(g^{k_j})\in\partialC f(x^*)$ for each $g^{k_j}$ so that\\ 
  $
  \|g^{k_j}-v(g^{k_j})\|\leq c_2\max\left\{\Delta_{k_j},\left\|x^{k_j} - x^*\right\|\right\}
  $
  with $c_2$ defined by~\eqref{eq:c_2}.
  By the acceptability of every iteration indexed by $k_j$,
  $
  \|g^{k_j}-v(g^{k_j})\|\leq c_2 \max\left\{\eta_2\|g^{k_j}\|,\left\|x^{k_j} - x^*\right\|\right\}
  $ holds,
  and so 
  $
  \|v(g^{k_j})\|\leq \max\left\{ (1 + c_2 \eta_2)\|g^{k_j}\| , \|g^{k_j}\| + c_2 \left\|x^{k_j} - x^*\right\| \right\}.
  $
  Moreover, since $\|g^{k_j}\|\to 0$ and $\left\|x^{k_j} - x^*\right\|\to 0$ by assumption, $\left\| v(g^{k_j}) \right\|
  \to 0$. Proposition~7.1.4 in~\cite{Facchinei2003} then yields the
  claimed result by establishing that $\partialC f$ is
  \emph{outer semicontinuous} and therefore $0 \in \partialC f(x^*)$. 
\end{proof}

\begin{theorem} \label{thm:cluster_linear}
  Let \assrefrange{f}{generators} hold. If
  $x^*$ is a cluster point of a sequence $\{x^k\}$ generated by
  \algref{manifoldsampling}, then $0\in\partialC f(x^*)$.
\end{theorem}

\begin{proof}
  First, suppose that there are only finitely many successful iterations, with
  $k'$ being the last.
  Suppose toward a contradiction that $0\notin\partialC f(x^{k'})$.
  By the same reasoning used to conclude \eqref{eq:activity_settles}, 
  there exists $\bar\Delta>0$ such that 
  $
  \mathbb{A}(F(x^{k'})) = \bigcup_{y \in \cB(F(x^{k'});L_F\Delta)} \mathbb{A}(y) \;
    \mbox{ for all } \Delta \le \bar{\Delta}.
    $

  By assumption,
  $\Delta_k$ decreases by a factor of $\gammad$ in each iteration after
  $k'$ since every iteration after $k'$ is unsuccessful. Thus there
  is a least iteration $k''\geq k'$ such that
  $\Delta_{k''} \le \bar\Delta$.  By \assref{generators}, for each
  $k\geq k''$, $\Act{\mathbb{Z}^k} = \Act{F(x^{k'})}$, and therefore 
  $( \nabla M(x^k) \nabla h_j(F(x^k))) \in\Gen^k$ for all $j\in\mathbb{A}(F(x^k))$. Since $k'$ is the last
  successful iteration, $x^{k} = x^{k'}$ for all $k \ge k'' \ge k'$. 
  Consequently, under \assrefrange{f}{flmodels},
  the conditions for \lemref{weird_v_approx} hold for $x\gets x^{k}$, $y \gets 
x^{k'}$ (noting that $x^k = x^{k'}$) $\Delta\gets 0$, $I \gets \Act{\mathbb{Z}^k}$, and
  $J \gets \Act{F(x^{k'})}$. Thus, for each $k\geq k''$, 
  $g^k\in\partialC f(x^{k'})$.  

  Since $0\notin\partialC f(x^{k'})$ by supposition,
  $v^*\defined\Proj(0,\partialC f(x^{k'}))$ is nonzero, and so
  \begin{equation}\label{eq:v_approx_conclusion_linear}
    \|g^k\|\geq\|v^*\|>0 \qquad
      \mbox{ for all } k\geq k''.
  \end{equation}
  Since $\Delta_k \to 0$, $\Delta_k$ will satisfy the conditions of
  \lemref{success_when_linear} for $k$ sufficiently large: there will be a
  successful iteration contradicting $k'$ being the last. 

  Next, suppose there are infinitely many successful iterations.
  We will demonstrate that there exists a subsequence of successful iterations
  $\{k_j\}$ that simultaneously satisfies both
  $x^{k_j}\to x^* \mbox{ and } \|g^{k_j}\|\to 0.$
  If the sequence $\{x^k\}_{k\in\mathbb{N}}$ converges, then the subsequence 
$\left\{ x^{k_j} \right\}_{j\in\mathbb{N}}$ from \lemref{g_to_0_linear}
  satisfies these two conditions. 
  Otherwise, if the sequence $\{x^k\}$ is not convergent, we will show that 
\linebreak[4] $\liminf_{k\to\infty}(
  \max\{\|x^k-x^*\|,\|g^k\|\})=0$ for each cluster point $x^*$. Suppose toward
  contradiction that there exists $\bar\nu>0$, an iteration $\bar{k}$,
  and a cluster point $x^*$ of the sequence $\{x^k\}$ such that
  $\left\{ x^k \right\}_{k \in \mathcal{K}}$ converges
  to $x^*$ and such that $\|g^k\| > \bar\nu$ 
  for all $k\in \mathcal{K}$,
  where
  $
    \mathcal{K} \defined \{k\colon k\geq\bar{k}, \|x^k-x^*\|\leq \bar\nu\}.
    $
  As an intermediate step in the combination of~\eqref{eq:nonincreasing_sequence} and~\eqref{eq:eta_delta_normg_linear},
  we had shown that
  $\ds\sum_{j=0}^\infty (\eta_1\frac{\kappad}{2}\|g^{k_j}\|\Delta_{k_j} - \bigC\Delta_{k_j}^2) < \infty.$
  Because~\eqref{eq:eta_delta_normg_linear} shows that $\sum_{j=0}^{\infty} \Delta_{k_j}^2$ is finite, 
  we may conclude that
  $\ds\sum_{j=0}^\infty \eta_1\frac{\kappad}{2}\|g^{k_j}\|\Delta_{k_j}  < \infty.$
Thus,
\begin{equation}
    \label{eq:finite_sum_linear}
    \begin{aligned}
      \eta_1\frac{\kappad}{2}\displaystyle\sum_{k\in \mathcal{K}} \|g^k\|\|x^{k+1}-x^k\| \leq &
      \eta_1\frac{\kappad}{2}\displaystyle\sum_{j=0}^\infty \|g^{k_j}\|\|x^{{k_j}+1}-x^{k_j}\| \\ \leq &
      \eta_1\frac{\kappad}{2}\displaystyle\sum_{j=0}^\infty \|g^{k_j}\|\Delta_{k_j}
      < \infty.
    \end{aligned}
  \end{equation}
  Because $\|g^k\|>\bar\nu$ for
  all $k\in \mathcal{K}$, we conclude from~\eqref{eq:finite_sum_linear} that
  \begin{equation}
    \label{eq:finite_sum2_linear}
    \displaystyle\sum_{k\in \mathcal{K}}\|x^{k+1}-x^k\| < \infty.
  \end{equation}

  Because $x^k\not\to x^*$, 
  for any choice of $\hat\nu\in(0,\bar{\nu})$
  the quantity
  \[
    q(k')\defined \min\{\kappa\in\mathbb{N}\colon \kappa>k',\quad \|x^{\kappa}-x^{k'}\| > \hat\nu\}
  \]
  is well defined for any $k'\in\mathcal{K}$. 
  For any $k'\in\mathcal{K}$, $\{ k',
  k' + 1, \ldots, q(k')-1\} \subset \mathcal{K}$.

  From~\eqref{eq:finite_sum2_linear}, there exists $N \in
  \mathbb{N}$ such that 
  $
  \ds\sum_{\substack{k \in \mathcal{K} \\ k \ge N}} \left\| x^{k+1} - x^k \right\| \le
  \hat{\nu}.
  $
  Letting $k' \ge N$ be arbitrary, we arrive at
  \begin{equation*}
    \label{eq:triangle_ineq_linear}
    \hat\nu < \|x^{q(k')}-x^{k'}\| \leq 
    \displaystyle\sum_{i\in\{k',k'+1,\dots, q(k')-1\}} \|x^{i+1}-x^i\| \le 
  \sum_{\substack{k \in \mathcal{K} \\ k \ge N}} \left\| x^{k+1} - x^k \right\| \le
  \hat{\nu} ,
  \end{equation*}
  a contradiction. 
  Thus, $\liminf_{k\to\infty}
  (\max\{\|x^k-x^*\|,\|g^k\|\})=0$ for all cluster points $x^*$.
  By \lemref{g_to_v_linear},
  $0\in\partialC f(x^*)$ for all such subsequences.
\end{proof}

\section{Numerical Experiments}\label{sec:tests}
We now present the performance of an implementation of
\algref{manifoldsampling} for problems of the form \eqref{eq:func_def}.

\subsection{Implementation details}\label{sec:implementation_details}
To study its practical efficiency, we produced a
MATLAB implementation of \algref{manifoldsampling}, which we denote
manifold sampling: general (MSG). 
We outline the specific choices made in our implementation.

We considered two versions of MSG, MSG-1 and MSG-2, 
which provide distinct approaches to 
initializing and updating $\mathbb{Z}^k$ in \linerefa{starting_gen_set}
and \linerefa{grow_Zk} of \algref{manifoldsampling}.
MSG-1 implements \linerefa{starting_gen_set} as
$\mathbb{Z}^k\leftarrow \{ F(x^k) \}$
and \linerefa{grow_Zk} as $\mathbb{Z}^k \gets \mathbb{Z}^k\cup\{z\}$,
while 
MSG-2 implements \linerefa{starting_gen_set} as
$
\mathbb{Z}^k\leftarrow \{ F(x^k) \}\cup \left(Y\cap\{ F(y)\colon y\in\cB(x^k;\Delta_k) \}\right)
$
and \linerefa{grow_Zk} as
$
\mathbb{Z}^k \gets \mathbb{Z}^k \cup \{z\} \cup \left(Y\cap \{F(y)\colon y\in\cB(x^k;\Delta_k) \}\right),
$
where, as in \secref{gensets}, $Y$ is the set of all $y\in\R^n$
previously evaluated during the current run of \algref{manifoldsampling}.

The default parameters of MSG are fixed to $\eta_1 = 0.01$, $\eta_2 = 10^4$,
$\kappad = 10^{-4}$, $\gamma_d = 0.5$, $\gamma_i = 2$, and $\Delta_{\max} =
10^8$.  
We remark that the selection of $\eta_2 = 10^4$ may violate the restriction on $\etamax$ specified 
in~\eqref{eq:eta_max}.  
The bound in \eqref{eq:eta_max} was derived for the sake of worst-case analysis
(see, e.g., \eqref{eq:nonincreasing_sequence_C_nonzero} and \eqref{eq:nonincreasing_sequence_C_zero}),
and we thus expect \eqref{eq:eta_max} to be an unnecessarily conservative restriction in the most general case.
Therefore, for the sake of labeling iterations acceptable more frequently and thus accepting potentially larger trial steps,
we relax the condition in \eqref{eq:eta_max}.
This motivates the addition of a safeguard to MSG;
the criterion in \linerefa{j_in_Act_test} is augmented to test both that $j \in
\mathbb{A}(\mathbb{Z}^k)$ and that  $h(F(x^k + s^k)) < h(F(x^k))$. 
 
We also include some
termination conditions in our implementation of \algref{manifoldsampling}.  
First, the
outer \textbf{for} loop (\linerefa{outer_for_loop}) is terminated if the number of
evaluations of $F$ has exceeded a fixed budget.
MSG also employs a
termination condition before \linerefa{acceptabilityCheck} such that if
$\|g^k\|\leq g_{tol}$ and $\Delta_k \leq \Delta_{\min}$
for some positive constants $g_{tol}$ and $\Delta_{\min}$, then MSG
terminates. 
We fixed $g_{tol}=\Delta_{\min}=10^{-13}$ in our experiments.

We use a MATLAB implementation of GQT~\cite{More553} to compute $s^k$ in our
trust-region subproblems \eqref{eq:trsp}.
We explicitly check whether \eqref{eq:fraction_cauchy_decrease} is satisfied by $s^k$; 
if it is not, then we employ the step prescribed by \lemref{approx_solve}
to ensure that \eqref{eq:fraction_cauchy_decrease} is satisfied. 

For the purposes of model building in \linerefa{build_models} and model updating
in \linerefa{update_models}, we employ the 
 minimum Frobenius norm quadratic interpolation and geometry point selection routines used in 
the implementation of POUNDERS~\cite{Wild14}.
As is frequently seen in practical implementations of trust-region methods, 
MSG additionally modifies the trust-region radius management beginning in \linerefa{successCheck}.
In particular, while trial steps are still accepted, provided $\rho_k > \eta_1$, 
the trust-region radius only increases provided $\rho_k > 0.5$.

\subsection{Test problems}

We benchmark our implementation of \algref{manifoldsampling} on objectives of the form
$f(x) = h(F(x))$, where $F:\R^n\to\R^p$ is derived from the functions in the
Mor\'e--Wild benchmarking test set~\cite{JJMSMW09}, which were originally intended for 
nonlinear least-squares minimization (that is, $h = \|\cdot\|_2^2$). 
This initially gives us 53 problems, as specified by combinations of definitions of $F$ and initial points $x^0$. 
In the test set, the functions $F$ are all differentiable, and all but four are nonconvex.
The dimension of $F$ ranges from $2$ to $65$. 

For our experiments, we define $h$ as a piecewise-quadratic function of the form
\begin{equation}\label{eq:max_quad_def}
  h(z) \defined \max_{j \in \{1,\ldots,l\}} \left\{ h_j(z) \defined \left\| z - z_j \right\|_{Q_j}^2 + b_j \right\}
\end{equation}
defined by $z_j \in \mathbb{R}^p$, $Q_j \in \mathbb{R}^{p \times
p}$, and $b_j \in \mathbb{R}^l$ for $2 \le l \in \mathbb{N}$.
We use the notation that, for a given matrix $Q$, $\left\|y\right\|_{Q} \defined y^\top Qy.$
For each of the 53 functions and starting-point pairs $(F,x^0)$, we 
generated a single random instance of \eqref{eq:max_quad_def} in the following manner.
We first set
$z^j \defined F(y^j)$ for $j\in \{1,\ldots,l\}$, where $y^j$ is drawn uniformly from the ball 
$\{ y\colon \left\| y - x^0 \right\|_{\infty} \le 20 \}$. 
We then randomly generate a positive-definite matrix 
$Q_1$ and negative-definite matrices $Q_2,\ldots,Q_l$.
The positive-definiteness of $Q_1$ ensures that $h\circ F$ is bounded from below. 
We set all $b_j$ to be 0 except $b_1$, which we define as
$
b_1 \defined -2\max_{j \in \{2,\ldots, l\}} \left\{ \left\| F(y^j) - F(y^1)
\right\|_{Q_1}^2 \right\}
$
in order to ensure that $h(F(y^j)) = 0$ for $j \in \{2,\ldots,l\}$. 
This definition of $b_j$ also guarantees that 
 $h_j(z)=h(z)$ for at least one value of $z$ (in particular, $z=F(y^j))$;
 intuitively, for each $j$, we are increasing the likelihood that
  $j\in\Act{F(x^k)}$ for some $x^k$
 evaluated during a given run of an optimization method. 
With this particular random construction, 
stationary measures ought to be small only in neighborhoods of kink points
or at the global maxima of the negative definite quadratics (the latter of which
are not local minima of $h(F(\cdot))$). 

For each $(F,x^0)$ pair in the Mor\'e--Wild benchmark set
 and for each value of $l \in \left\{ 2p, 4p, 8p, 16p  \right\}$ 
 we repeat our random generation scheme five times. This procedure produces
$53\times4\times5=1060$ benchmarking problems. 

\begin{figure}[t]
	\centering
  \hfil
  \includegraphics[width=0.3\textwidth]{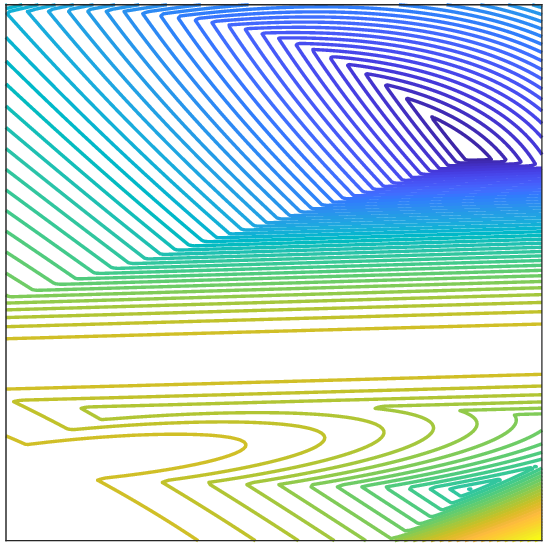}
  \hfil
  \includegraphics[width=0.3\textwidth]{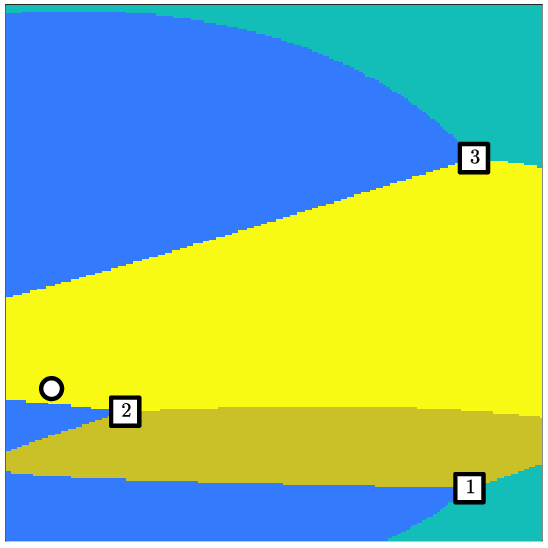}
  \hfil
  \caption{Contour plot (left) and manifold information (right) for one of our test
  functions. The circle shows the starting point, and the squares show the three
  points converged to by the five methods used in our numerical experiments.}
  \label{fig:contour_plot}
\end{figure}

\figref{contour_plot} shows the contour plot of one of our test problems.
We highlight that these problems have multiple potential Clarke stationary
points that have varying values of $h\circ F$.
For this reason, we must compare the performance of different implementations
not only in terms of the best function value obtained but also in terms of a metric designed
to approximate a Clarke stationarity measure. 

\subsection{Comparing performance}
We consider two quantities of interest when comparing methods for solving
problems of the form \eqref{eq:max_quad_def}: the objective value $f$ and approximate
stationary value $\Gamma$. 
With respect to $f$,
we consider a method to have solved a problem $p$
to a level $\tau$ after $t$ function
evaluations, provided the corresponding point $x^t$ satisfies
\begin{equation} \label{eq:f_converged}
  f(x^0) - f(x^t) \ge (1 - \tau) \left( f(x_0) - f^* \right),
\end{equation}
where $x^0$ is a starting point common to all methods and $f^{*}$ is the best-found
function value for all methods being compared. That is, we consider a
problem solved with respect to $f$ when it has found more than $(1- \tau)$ of
the most decrease from
$x^0$ found by any method being compared.

Determining an approximate stationary value $\Gamma(x^t)$ at a point $x^t$
evaluated by a method requires more care. To do so, we randomly generate $50$
points \linebreak[4] $\mathcal{S}^t\subset\cB(x^t,10^{-5})\subset\R^n$ and then compute 
\begin{equation}
\label{eq:gradient_bundle}
\mathcal{G}(x^t)\defined \{\nabla F(s) \nabla h_j(F(s))\colon j\in\Act{F(s)}, \; s\in\mathcal{S}^t\}.
\end{equation}
We then define $\Gamma(x^t) \defined \Proj(0,\cop{\mathcal{G}(x^t)})$. 
The gradient values in \eqref{eq:gradient_bundle} can be computed (in
postprocessing) because $\nabla F$ is computable in closed
form for the problems considered and each $h_j$ is a quadratic function by construction.
We consider a problem to be solved to a level $\tau$  with respect to
$\Gamma(x^t)$ when the minimum-norm element of the convex hull of the sample of gradients
$\mathcal{G}(x^t)$ is less than $\tau$. That is, 
\begin{equation} \label{eq:Gamma_converged}
  \Gamma(x^t) \le \tau. 
\end{equation}

We use data profiles~\cite{JJMSMW09} to compare the performance of methods
for nonsmooth optimization using the problems and metrics defined above.
To construct data profiles, we determine how many evaluations of $F$ are
required by each method to solve a given problem to a level $\tau$ for either 
criterion \eqref{eq:f_converged} or criterion \eqref{eq:Gamma_converged}.
Once a method satisfies the given criterion on any problem for the first time after $t$
evaluations of $F$, its data profile line is incremented by $\frac{1}{1060}$ at
the point $\frac{t}{(n_p+1)}$ (where $n_p$ is the dimension of the problem)
on the horizontal axis. The data profile therefore shows the cumulative fraction
of problems solved by each method as a function of the number of evaluations of $F$
(scaled by $n_p + 1$).

\subsection{Utilizing nearly active manifolds}
In preliminary experiments, we identified the following practical modification that may
be made to an implementation of \algref{manifoldsampling}:
it can be useful to slightly alter the definition of $\Act{z}$.
For example, given $h$ of the form \eqref{eq:max_quad_def}, 
consider a
$z$ such that $h(z) = h_1(z)$ but $\left| h_2(z)
- h_1(z) \right|\approx 0$.
Although this is generally insufficient evidence to conclude
the existence of $z'$ in a neighborhood of $z$ such that 
$h(z')=h_2(z')$, it could potentially be 
beneficial for an implementation of \algref{manifoldsampling}
to allow $\Act{z} = \{1,2\}$ instead of $\Act{z}=\{1\}$.  
In the event that a $z'$ does exist realizing $\Act{z'}=\{2\}$, 
having $\Gen^k$ include this phantom information may permit 
a manifold sampling loop to terminate earlier than it would have otherwise. 

Altering the definition of $\Act{z}$ 
will not affect the theoretical convergence of our algorithm, provided 
$\Act{F(x^{k})}\subseteq\Act{F(x^*)}$ as $\Delta_k \to 0$. 
To demonstrate the effects of changing the definition of $\Act{z}$ by including \emph{nearly active} 
manifolds in the definition of activity for problems of the form \eqref{eq:max_quad_def}, we
introduce a parameter $\sigma$ 
and consider instead
$\mathbb{A}_{\sigma,\Delta_k}(z) \defined \left\{ j: |h(z) - h_j(z)| \leq \min\{\sigma,\Delta\}\right\}.$

\begin{figure}[t]
	\centering
  \includegraphics[width=0.47\textwidth]{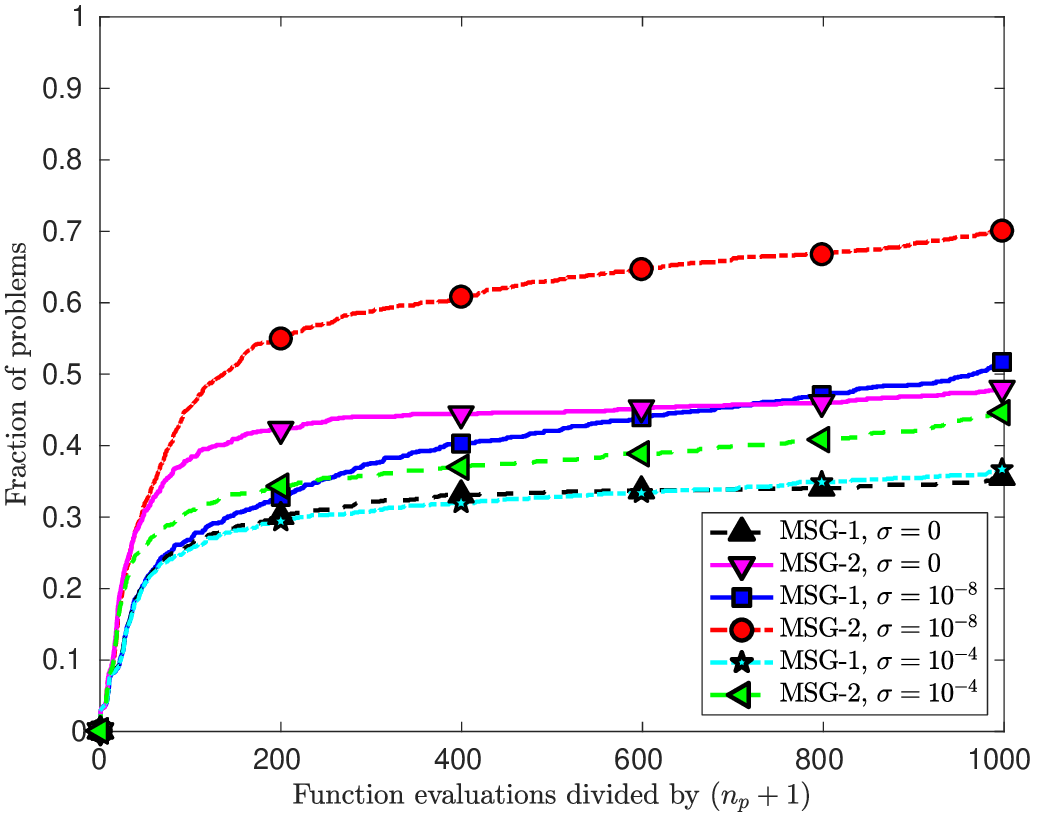}\hfil
  \includegraphics[width=0.47\textwidth]{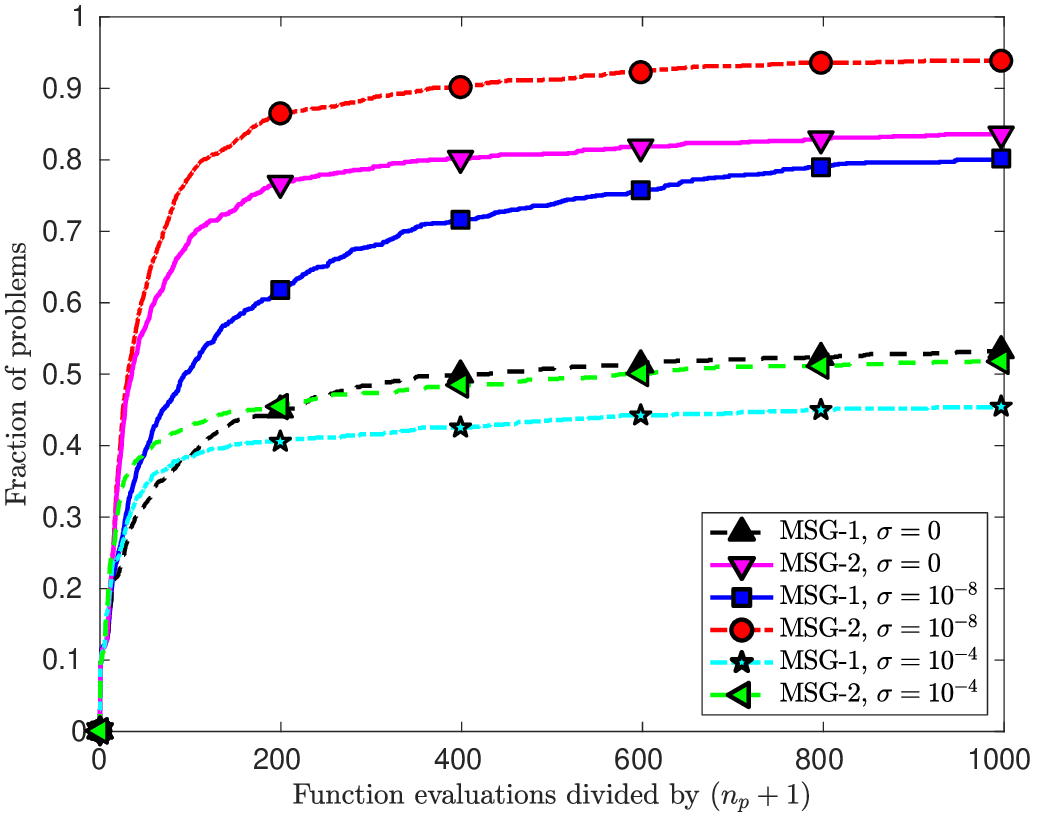}
  \caption{Data profiles using function values (left) and approximate stationary measure $\Gamma$ (right) for $\tau = 10^{-3}$.}
	\label{fig:h_act_tol}
\end{figure}
We show in 
\figref{h_act_tol}
data profiles of both MSG-1 and MSG-2, replacing the definition of $\Act{z}$ with $\mathbb{A}_{\sigma,\Delta_k}(z)$ for values
$\sigma \in \{ 0,10^{-4},10^{-8}\}$. Notice that $\mathbb{A}_{0,\Delta_k}(z) = \Act{z}$ as defined in \defref{pc1manifold}.

While investigating the relatively worse behavior of MSG-1, we found that (perhaps unsurprisingly)
many iterations are spent
rediscovering manifolds that had been identified on previous
iterations. On the other hand, MSG-2, with its memory of recent manifolds encoded in $Y$, 
begins iterations with more of the manifold information it needs to find descent.

We observe a marked improvement in increasing $\sigma$
from 0 to $10^{-8}$.
When $\sigma=0$, a sample point may be close to---but not exactly on---a
place where multiple quadratics define $h$. 
We hypothesize that setting $\sigma$ to a small but nonzero
value (in this case, $10^{-8}$) allows MSG to exploit knowledge of 
multiple nearly active manifolds near kinks, which are the locations of stationary points
of our test set by construction. 
However,
setting $\sigma$ to be too large degrades performance, likely because of too many
inactive manifolds being used by MSG.
Based on this initial tuning, we set the parameter $\sigma=10^{-8}$ for both MSG-1 and MSG-2
throughout the remainder of our numerical experiments, thus replacing
$\Act{z}$ with $\mathbb{A}_{10^{-8},\Delta_k}(z).$

\subsection{Comparisons with other methods}

\begin{figure}[t]
	\centering
  \includegraphics[width=0.45\textwidth]{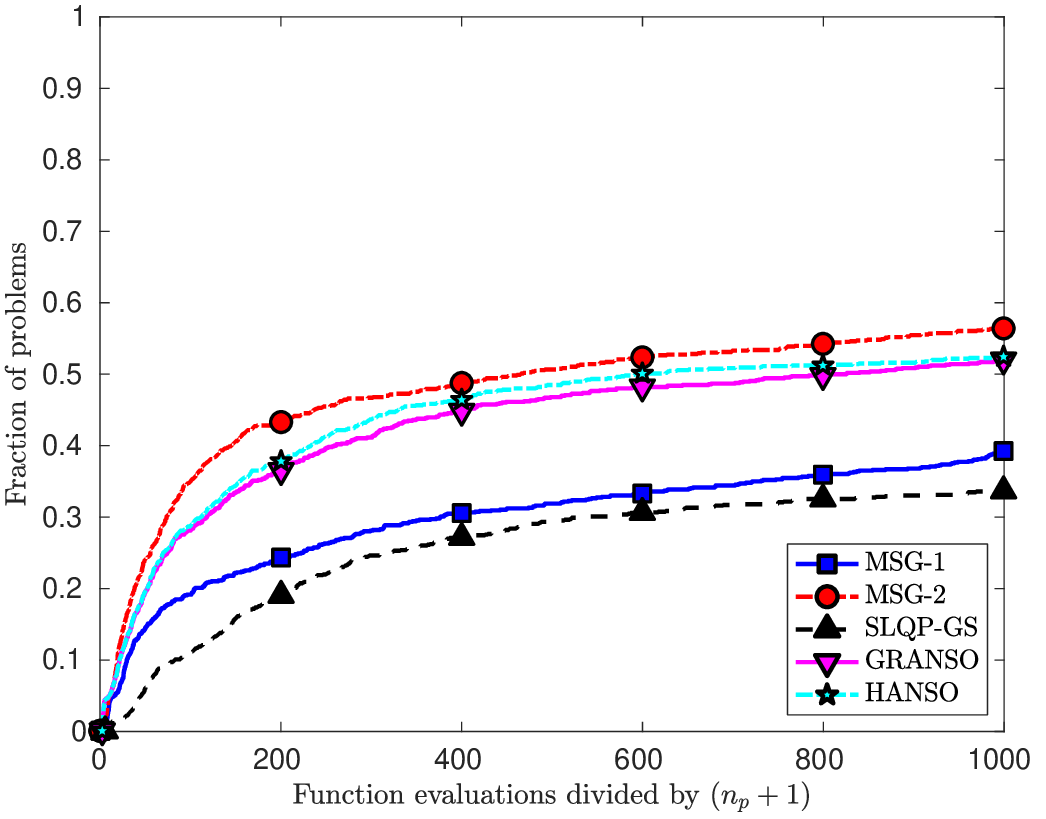} \hfil \includegraphics[width=0.45\textwidth]{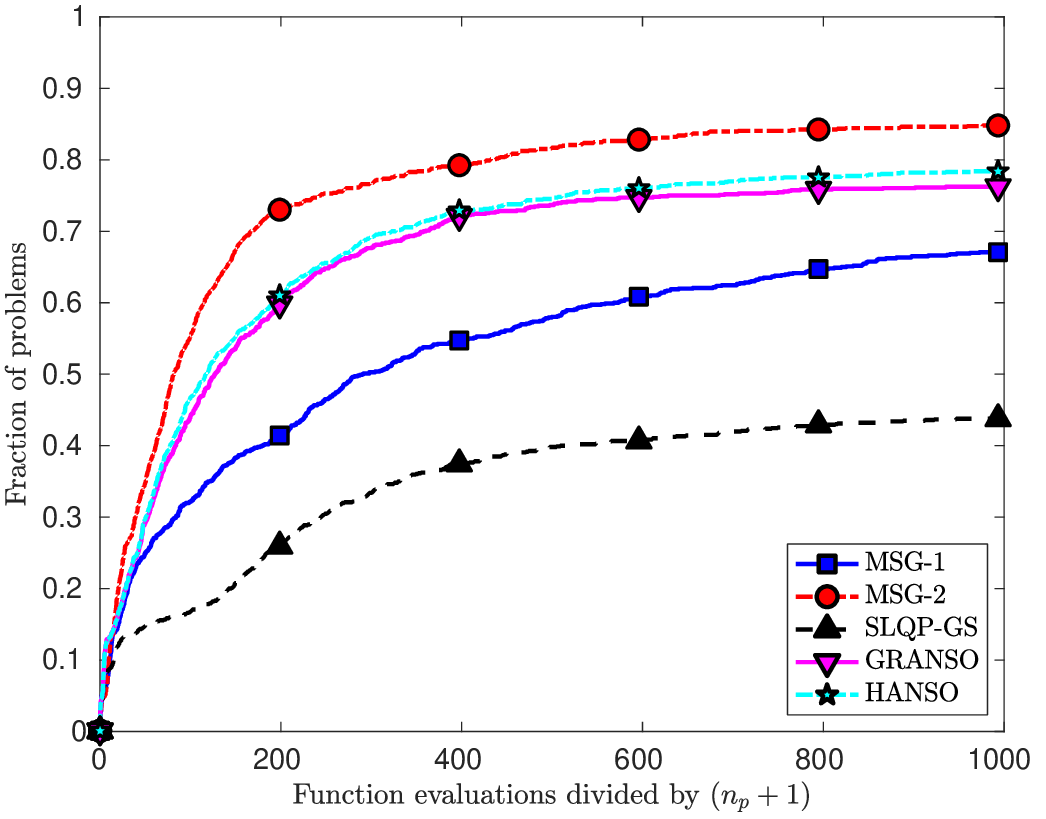}
	\caption{Data profiles using function values with $\tau = 10^{-3}$ (left) and using the approximate stationary measure $\Gamma$ for $\tau = 10^{-5}$ (right).}
	\label{fig:experiments_grad}
\end{figure}
We compare MSG with other nonsmooth optimization methods
that require an oracle for the computation of $\nabla F$. 
These are
GRANSO, HANSO, and SLQP-GS. 
The GRadient-based Algorithm for Non-Smooth
Optimization (GRANSO) employs a sequential quadratic optimization method with
Broyden--Fletcher--Goldfarb--Shanno (BFGS) approximate Hessian
updates~\cite{curtis2017bfgs}. GRANSO was run with its default settings.
The Hybrid Algorithm for Non-Smooth
Optimization (HANSO)
implements BFGS and gradient sampling methods together with a bundle
method~\cite{overton2009hanso}. HANSO is hard-coded to perform at most 100
gradient sampling iterations since such iterations can be expensive; this cap was
removed.
Sequential Linear or Quadratic
Programming
with Gradient Sampling (SLQP-GS) has two modes: the sequential linear programming mode
and the sequential quadratic programming mode, where gradient information is
always obtained by a sampling process to compute search directions
efficiently~\cite{curtis2012sequential}. We use the sequential quadratic mode in our experiments.

\figref{experiments_grad} (left) shows a data profile in terms of 
decrease in $f$; that is, the definition of solved is determined by~\eqref{eq:f_converged}. 
As previously remarked, however, function values may be insufficient to describe the performance 
of methods on our benchmarking test set, and so 
the right plot in
\figref{experiments_grad} shows data profiles in terms of 
$\Gamma$, and the definition of solved is determined by~\eqref{eq:Gamma_converged}.

From \figref{experiments_grad}, we may conclude that MSG-2 outperforms MSG-1.
It is remarkable that without requiring any values of $\nabla F$, MSG-2
exhibits competitive performance with HANSO and GRANSO.  
We also note that
MSG-1 and MSG-2 both outperform SLQP-GS in our experiments.

\section{Discussion}\label{sec:conclusion}
We note that the objective function in~\eqref{eq:func_def} could involve an additional
summand $\psi(x)$ (that is, we could redefine $f(x)\defined \psi(x)+h(F(x))$) 
for some $\psi\colon\R^n\to\R$ assumed continuously differentiable and bounded below.
Our analysis could easily be extended to apply to such functions.

One naturally desires a worst-case complexity rate for manifold sampling
algorithms. While such analysis may be possible, it would rely critically on the
per-iteration cost of the manifold sampling loop identifying the selection
functions active in the current trust region. For the worst case, one can
construct examples where all selection functions need to be identified. Similar
concerns may explain why worst-case complexity results have not yet
been demonstrated for gradient sampling methods.

\appendix

\section{Additional Results}\label{sec:appendix}
Here we provide an additional lemma as a reference for the proofs of the lemmas in \secref{prelims}.

\begin{lemma}\label{lem:directional_derivative_inf}
	Let \assref{h} hold. Let $z(\alpha)$ be as in~\eqref{eq:zalpha}.  Then\\
	$
	h(F(x)) - h(F(x+s)) \geq \int_0^1 \inf_{j\in\mathbb{A}(z(\alpha))}\{ \nabla h_j(z(\alpha))^\trans (F(x) - F(x+s)) \} d\alpha.
	$
\end{lemma}
\begin{proof}
	By Assumptions~\ref{ass:h}.\ref{subass:h}~and~\ref{ass:h}.\ref{subass:Ieh}, we can define sets $\mathcal{L}_1,\ldots,\mathcal{L}_N$ by\\
  $
	\mathcal{L}_i = \{ \alpha\in [0,1)\colon \exists \tau> 0 \text{ such that } i = \min\{j\colon j\in\mathbb{A}(z(\beta)) \} \; \forall \beta\in [\alpha,\alpha+\tau) \}.
  $
	
	By definition, it is immediate that $\{\mathcal{L}_i\}_{i=1}^N$ form a partition of $[0,1)$.
	
	For any interval $[q,r)\subseteq [0,1)$ such that $[q,r)\subseteq \mathcal{L}_i$ for some $i=1,\ldots,N$,
	\begin{equation*}
	\begin{aligned}
	&h(z(r)) - h(z(q))\\
	= &h_i(F(x+s) + r(F(x) - F(x+s))) - h_i(F(x+s) + q(F(x) - F(x+s))) \\
	= &\int_q^r \nabla h_i(z(\alpha))^\trans (F(x) - F(x+s)) d\alpha \\
	=  &\int_q^r \inf_{i\colon \alpha\in\mathcal{L}_i} \{ \nabla h_i(z(\alpha))^\trans (F(x) - F(x+s)) \} d\alpha \\
	\geq &\int_q^r \inf_{i\in\mathbb{A}(z(\alpha))} \{ \nabla h_i(z(\alpha))^\trans (F(x) - F(x+s)) \} d\alpha.
	\end{aligned}
	\end{equation*}
	Define a \emph{maximal interval} as any half-open interval $[q,r)\subset\mathcal{L}_i$ such that 
	$[q-\tau,r) \not\subset \mathcal{L}_i$ and $[q,r+\tau)\not\subset \mathcal{L}_i$ for all $\tau > 0$.
	Let $\Lambda_i$ denote the union of maximal half-open intervals $[q_{\ell},r_{\ell})\in\mathcal{L}_i$.
	Then,
	\begin{equation*}
	\begin{aligned}
  h(F(x)) - h(F(x+s)) 
	= &\sum_{j=1}^N\sum_{\ell\in\Lambda_j} h(z(r_{\ell})) - h(z(q_{\ell})) \\
	\geq &\sum_{j=1}^N\sum_{\ell\in\Lambda_j}\int_{q_{\ell}}^{r_{\ell}} \inf_{i\in\mathbb{A}(z(\alpha))} \{ \nabla h_i(z(\alpha))^\trans (F(x) - F(x+s)) \} d\alpha \\
	= &\int_0^1 \inf_{i\in\mathbb{A}(z(\alpha))} \{ \nabla h_i(z(\alpha))^\trans (F(x) - F(x+s)) \} d\alpha.
	\end{aligned}
	\end{equation*}
\end{proof}

\section*{Bisection search algorithm}
Here we provide a bisection search algorithm as an alternative to
\algref{gridsearch}
for use in \linerefb{linearization} in
\algref{manifoldsampling}.
Although we have not been able to prove a result analogous to \lemref{candidate_search_lemma},
\algref{bisecsearch} is the search algorithm that we implement in practice. 
In our numerical experiments, \algref{bisecsearch} always terminated successfully. 

\begin{algorithm2e}[h]
  \fontsize{8}{8}\selectfont
	\caption{Bisection Search for $\nabla h_j(z)$ \label{alg:bisecsearch}}
  \DontPrintSemicolon 
	\SetAlgoNlRelativeSize{-5}
	\SetKw{true}{true}
	\SetKw{break}{break}
	\lIf{$F(x)$ and some $j$ satisfy~\eqref{eq:baoyu_condition}}{return $\nabla h_j(F(x))$} 
	\lIf{$F(x+s)$ and some $j$ satisfy~\eqref{eq:baoyu_condition}}{return $\nabla h_j(F(x+s))$} 
	Set $\alpha \gets 0$, $\beta \gets 1$
	
	\While{\true}
	{ Set $z(\frac{\alpha+\beta}{2}) \gets \frac{\alpha+\beta}{2}F(x) + (1-\frac{\alpha+\beta}{2})F(x+s)$\\
		\lIf{$z(\frac{\alpha+\beta}{2})$ and some $j$ satisfy~\eqref{eq:baoyu_condition}}{return $\nabla h_j(z(\frac{\alpha+\beta}{2}))$}
		{
			\uElseIf{$h(z(\frac{\alpha+\beta}{2})) > \frac{\alpha+\beta}{2}h(F(x)) + (1-\frac{\alpha+\beta}{2})h(F(x+s))$}
			{$\beta\gets \frac{\alpha+\beta}{2}$}
      \Else{$\alpha\gets \frac{\alpha+\beta}{2}$}
		}
	}
\end{algorithm2e}

\section*{Acknowledgments}
This work was supported by the U.S.~Department of Energy, Office of Science,
Advanced Scientific Computing Research, under Contract DE-AC02-06CH11357.
Support for this work was also provided through the SciDAC program funded by the
U.S.~Department of Energy, Office of Science, Advanced Scientific Computing Research.
We thank Tim Mitchell for assistance in improving the performance of GRANSO and HANSO.
We are grateful for the comments from the three anonymous reviewers that 
improved an early version of this manuscript.

\FloatBarrier
\bibliographystyle{siamplain}
\bibliography{../../bibs/refs}

\vfill
\framebox{\parbox{.90\linewidth}{\scriptsize The submitted manuscript has been created by
UChicago Argonne, LLC, Operator of Argonne National Laboratory (``Argonne'').
Argonne, a U.S.\ Department of Energy Office of Science laboratory, is operated
under Contract No.\ DE-AC02-06CH11357.  The U.S.\ Government retains for itself,
and others acting on its behalf, a paid-up nonexclusive, irrevocable worldwide
license in said article to reproduce, prepare derivative works, distribute
copies to the public, and perform publicly and display publicly, by or on
behalf of the Government.  The Department of Energy will provide public access
to these results of federally sponsored research in accordance with the DOE
Public Access Plan \url{http://energy.gov/downloads/doe-public-access-plan}.}}
\end{document}